 \documentclass[11pt,letterpaper]{amsart}

 \usepackage{txfonts} 

 \usepackage[colorlinks,linkcolor=blue,citecolor=blue]{hyperref}

 \theoremstyle{plain}
 
\newtheorem{theorem}{Theorem}[section]
\newtheorem{lemma}[theorem]{Lemma}
\newtheorem*{claim}{Claim}
\newtheorem{problem}[theorem]{Problem}

\theoremstyle{definition}
\newtheorem{definition}[theorem]{Definition}
\newtheorem{corollary}[theorem]{Corollary}
\newtheorem{proposition}[theorem]{Proposition}

\theoremstyle{remark}
\newtheorem{remark}[theorem]{Remark}

\numberwithin{equation}{section}

\begin{document}


\title[Fully nonlinear equation and its applications]{The partial uniform ellipticity and prescribed problems on the conformal classes of complete metrics}


\author{Rirong Yuan }
\address{School of Mathematics, 
	South China University of Technology, 
	Guangzhou 510641, China}
\email{yuanrr@scut.edu.cn}

 \keywords{Fully nonlinear equation; conformal transformation; fully nonlinear Loewner-Nirenberg problem; complete noncompact version of fully nonlinear Yamabe problem; symmetric concave function; partial uniform ellipticity}

\date{}

\dedicatory{}

 \begin{abstract}
	
	
	We clarify how close a second order fully nonlinear equation can come to uniform ellipticity, through 
	counting large eigenvalues of the linearized operator. 
This suggests an effective and novel
  way to understand the structure of fully nonlinear equations of elliptic and parabolic type.
	As applications, we solve a  
	 fully nonlinear version of the Loewner-Nirenberg problem and a noncompact complete version of fully nonlinear Yamabe problem. 
 Our method is delicate as shown by a topological obstruction.
	

\end{abstract}

\maketitle


  
  \section{Introduction}  
  
  
  \subsection{Partial uniform ellipticity}

  In their pioneering
  paper \cite{CNS3}, Caffarelli-Nirenberg-Spruck initiated the study of a fully nonlinear equation 
  \begin{equation} 
  	\label{equation-cns} 
  	\begin{aligned}
  		F(D^2 u) =\psi \mbox{  } \mbox{ in } \Omega\subset\mathbb{R}^n. 
  	\end{aligned} 
  \end{equation}
  This equation has the form 
  $F(D^2 u)=f(\lambda(D^2 u))$, where
 $\lambda(D^2u)=(\lambda_1,\cdots,\lambda_n)$ are the eigenvalues of $D^2 u$.
  As suggested by Caffarelli-Nirenberg-Spruck, $f$ is  a smooth symmetric function of $n$ real variables defined in an open symmetric  convex cone 
  $\Gamma\subset\mathbb{R}^n$,
  with vertex at the origin and 
   boundary $\partial \Gamma\neq\emptyset$,  
  containing the positive cone $\Gamma_n$. In addition,
  $f$ satisfies some fundamental structural conditions including 
  \begin{equation}
  	\label{concave}
  	f \mbox{ is  concave in } \Gamma,
  \end{equation}
  \begin{equation}
  	\label{elliptic}
  	\begin{aligned}
  		f_{i}(\lambda):= 
  		\frac{\partial f}{\partial \lambda_{i}}(\lambda)
  		> 0  \mbox{ in } \Gamma, \mbox{  } \forall  1\leqslant i\leqslant n.
  	\end{aligned}
  \end{equation}
  
  In some cases, one may 
  replace \eqref{elliptic} by the following weaker condition
  \begin{equation}
  	\label{elliptic-weak}
  	\begin{aligned}
  		f_{i}(\lambda) \geqslant 0,  
  		\mbox{  } \forall 1\leqslant i\leqslant n,
  	 	\mbox{  } \sum_{i=1}^n f_i(\lambda)>0,   \mbox{ in } \Gamma.
  	\end{aligned}
  \end{equation} 

  This includes among others Poisson equation and Monge-Amp\`ere equation, on which 
  there are substantial literature. 
  Other typical examples  
  are as follows:  
  \begin{equation}		\begin{aligned}  			
  	f=  (\sigma_k/\sigma_l)^{1/(k-l)},  \mbox{ } \Gamma=\Gamma_k, \nonumber
 	 \,	\mbox{ or }\, f = (\sigma_k/\sigma_l)^{1/(k-l)} \circ P_{n-1},   \mbox{ }    
 	  \Gamma= \mathcal{P}_k 	\nonumber  			  \end{aligned}  	\end{equation}   
  for  $0\leqslant l<k\leqslant n$, 
  where 
  $\sigma_k$ is the $k$-th elementary symmetric function,  $\Gamma_k$ is  the $k$-th G{\aa}rding cone,  
  and  $P_{n-1}(\lambda)= \sum_{i=1}^n \lambda_i \vec{\bf 1}-\lambda,$
  $\vec{\bf 1}=(1,\cdots,1)\in \mathbb{R}^n,$ 
   $\mathcal{P}_k=
  \{\lambda: P_{n-1}(\lambda)\in\Gamma_k\}$.
  
  We first look at the structure of  equation \eqref{equation-cns}. 
  The eigenvalues of coefficient matrix
  of the linearlized operator at $u$
  are precisely given by
   \[f_1(\lambda), \cdots, f_n(\lambda), \mbox{ for } \lambda= \lambda(D^2 u).\]
  As a consequence, \eqref{elliptic} yields that \eqref{equation-cns} is elliptic at $u\in C^2(\bar \Omega)$ satisfying $\lambda(D^2u)\in \Gamma$,  while \eqref{concave} implies that the operator $F(D^2u)=f(\lambda(D^2u))$ is concave with respect to $D^2u$ when $\lambda(D^2u)\in\Gamma$,
  thereby Evans-Krylov theorem \cite{Evans82,Krylov83} applies  
  once the bound  of real Hessian has been proved.  

  As is well known, the uniform ellipticity plays central roles in the study of second order elliptic equations, 
  especially in the theory of \textit{a priori} estimates, see e.g. \cite{GT1983}.
  However,  
  the uniform ellipticity of \eqref{equation-cns} breaks down in general,
  thereby causing various serious difficulties. 
  Consequently,  
  it is of importance to 
  determine when a fully nonlinear equation 
  becomes uniformly elliptic.
  
  This leads to the concept of partial uniform ellipticity. 
  \begin{definition}
  	\label{def-PUE}
  	We say that
  	$f$ is of \textit{$\mathrm{m}$-uniform ellipticity} in $\Gamma$,  
  	if $f$ satisfies 
  	\eqref{elliptic-weak} and 
  	there is a uniform positive constant $\vartheta$ 
  	such that for any $\lambda\in \Gamma$ with 
  	$\lambda_1\leqslant \cdots\leqslant \lambda_n$, 
  	\begin{equation}
  		\label{partial-uniform2}
  		\begin{aligned}
  			f_{i}(\lambda)\geqslant \vartheta\sum_{j=1}^n f_j(\lambda), \mbox{ } \forall 1\leqslant i\leqslant \mathrm{m}. 
  		\end{aligned}
  	\end{equation}
  	In particular, \textit{$n$-uniform ellipticity} is also called
  	\textit{fully uniform ellipticity}.

  	Accordingly, we have a similar notion of partial uniform ellipticity for a second order 
  	equation, 
  	if its linearized operator 
  	satisfies a similar condition.
  \end{definition}
There  naturally arises a problem: 
  \begin{problem}
  	\label{problem1-yuan}
  To 
  determine the integer $\mathrm{m}$ from \eqref{partial-uniform2} in
  	Definition \ref{def-PUE}.
  \end{problem} 
  
  The first part of this paper is devoted to 
  this problem.
  For any $f$ 
  satisfying 
  \begin{equation}
  	\label{addistruc}
  	\begin{aligned}
  		\lim_{t\rightarrow +\infty}f(t\lambda)>f(\mu) \mbox{ for any } \lambda, \mbox{ }\mu\in \Gamma,
  	\end{aligned}
  \end{equation}
  we observe that a  
  criterion, 
  extending the one 
   proposed initially in 
   \cite{yuan-regular-DP}
  leads 
  to bridging partial uniform ellipticity 
  and
  the maximum count of 
  zero
  components of vectors contained in $\Gamma$.
  To this end, we introduce an integer for $\Gamma$.
  \begin{definition}
  	\label{yuan-kappa}
  	For the cone $\Gamma$, 
  	we define    
  \begin{equation}\begin{aligned}	 {\kappa}_{\Gamma}=\max \left\{k: ({\overbrace{0,\cdots,0}^{k}},{\overbrace{1,\cdots, 1}^{n-k}})\in \Gamma \right\}. \nonumber 	\end{aligned}\end{equation}
  \end{definition}

 Below we state our main result 
  on the partial uniform ellipticity. 
  \begin{theorem}
  	\label{yuan-k+1}
  	Suppose  $(f,\Gamma)$ satisfies 
  	\eqref{concave} and \eqref{addistruc}.
  	Then we have \eqref{elliptic-weak} and that there exists a universal positive constant $\vartheta_{\Gamma}$ depending only on 
  	$\Gamma$, such that
  	for any $ \lambda\in \Gamma$ 
  	with $\lambda_1 \leqslant \cdots \leqslant\lambda_n$,
  	\begin{equation}
  		\begin{aligned}
  			f_{{i}}(\lambda) \geqslant   \vartheta_{\Gamma} \sum_{j=1}^{n}f_j(\lambda), \mbox{  } \forall 1\leqslant i\leqslant \kappa_\Gamma+1. \nonumber
  		\end{aligned}
  	\end{equation}
  	In particular, $f$ is of fully uniform ellipticity  
  	 in the type 2 cone.
  \end{theorem}

  \begin{remark}  	A choice of $\vartheta_{\Gamma}$   can be found 	in Remark \ref{remark111}.  	
  	 In Proposition \ref{yuan-kappa-2} we prove  $\kappa_\Gamma$ is in effect the maximum count of  negative 	components of vectors  in $\Gamma$.
     \end{remark}

  This theorem asserts that $f$ is 
  of \textit{$(\kappa_\Gamma+1)$-uniform ellipticity} in $\Gamma$,
  which is in effect sharp as shown in Corollary \ref{thm-sharp} below.   This is one of the main contributions we make in this paper. 
The theorem  applied to the case  $(f,\Gamma)=(\sigma_k^{1/k},\Gamma_k)$ 
gives back  \cite[Theorem 1.1]{Lin1994Trudinger} of Lin-Trudinger  
 with a different  method, 
 essentially using specific properties of $\sigma_k^{1/k}$ which cannot  be adapted to general functions. 

  Theorem \ref{yuan-k+1} suggests an effective and  novel
  way to understand the structure of fully nonlinear equations of elliptic and parabolic type.
As a by-product, we can confirm the following inequality in general context
  \begin{equation}
  	\label{key13-yuan}
  	\begin{aligned}
  		f_i(\lambda)\geqslant \vartheta_{\Gamma} \sum_{j=1}^n f_j(\lambda)   \mbox{ if } \lambda_i\leqslant 0. \nonumber
  	\end{aligned}
  \end{equation}  
  This inequality   
  was imposed as a key assumption by Li \cite{LiYY1991},  
 subsequently
  by  
  Trudinger \cite{Trudinger90}, Guan-Spruck  \cite{Guan1991Spruck} and Sheng-Urbas-Wang \cite{ShengUrbasWang-Duke} to study Weingarten equations.  Also, it can be applied to various  geometric PDEs from conformal geometry and real Hessian type fully nonlinear equations, see e.g.  \cite{SChen2007,SChen2009,Urbas2002,Guan12a}.

 Another remarkable 
 consequence 
 is that 
  a type 2 cone (see Definition \ref{def-type2}) guarantees the uniform ellipticity of the corresponding
   equations. 
   This allows  
us to study a  class of fully 
nonlinear equations, with applications to a
   fully nonlinear version of Loewner-Nirenberg problem 
   and a complete noncompact version of fully nonlinear Yamabe problem.

 \subsection{Fully nonlinear equations on Riemannian manifolds}
  Let $(M,g)$ be a connected Riemannian manifold of dimension $n$  with Levi-Civita connection $\nabla$. Let $\partial M$ denote the boundary, and $\bar M=M\cup\partial M$.
  Notice $\bar M=M$ if $\partial M=\emptyset$.
  
   We denote $\Delta u$, $\nabla^2 u$ and $\nabla u$  the Laplacian, Hessian and gradient of $u$ with respect to $g$, respectively. 
 For $\mathrm{U}$ a symmetric $(0,2)$-tensor, $\lambda(g^{-1} \mathrm{U})$ denote the eigenvalues of $\mathrm{U}$ with respect to $g$, 
  and the meaning of $\lambda(\tilde{g}^{-1}\mathrm{U})$ is obvious.

We apply partial uniform ellipticity to study a fully nonlinear equation
 \begin{equation}
 	\label{mainequ-1general}
 	\begin{aligned}
 		f(\lambda(g^{-1}(\Delta u \cdot g-\varrho\nabla^2u+A(x,\nabla u))))=\psi(x,u),
 	\end{aligned}
 \end{equation}
with a constant $\varrho$ subject to
\begin{equation}
	\label{assumption-4}
	\begin{aligned}
		\varrho<\frac{1}{1-\kappa_\Gamma \vartheta_{\Gamma}} \mbox{ and } \varrho\neq 0,
	\end{aligned}
\end{equation}
where 
  the right-hand side  
is a smooth positive  function 
such that   \begin{equation}
	\label{psi-gamma1}
	\begin{aligned}
		\psi_z(x,z)>0, \mbox{ } \lim_{z\to-\infty}\psi(x,z)=0, \mbox{ } 
		\psi(x,z)\geqslant h(x) e^{\alpha(x)z},
		 \mbox{ }  (x,z)\in  \bar M\times \mathbb{R}
	\end{aligned}
\end{equation}
holds with positive continuous functions $h$ and $\alpha$. 
 In addition, 
 \begin{equation}
 	\label{sup-infty}
 	\sup_{\Gamma}f=+\infty,
 \end{equation}
 \begin{equation}
 	\label{homogeneous-1-buchong2}
 	\begin{aligned}
 		f >0 \mbox{ in } \Gamma, \quad f =0 \mbox{ on } \partial \Gamma,
 	\end{aligned}
 \end{equation}
and
 $A(x,p)$ is  a smooth symmetric  $(0,2)$-tensor   satisfying 
for any $(x,p)\in T \bar M,$ 
\begin{equation}
	\label{R-gamma1}
	\begin{aligned}
		\lim_{|p|\to+\infty}	\left(\frac{|A(x,p)|}{|p|^2} + \frac{|D_pA(x,p)|}{|p|}\right)\leqslant\eta(x), 	  \mbox{ }		\lim_{|p|\to+\infty} \frac{|\nabla' A(x,p)|}{|p|^3} =0,
	\end{aligned}
\end{equation}
where  
$\eta$ is a positive
 continuous function, 
$\nabla' A$ denotes  partial covariant derivative of  $A$,  the meaning  of $D_pA$ is obvious when viewed as depending on $x\in  \bar M$.

 The equations of this form closely connect  with prescribed curvature equations  
 from conformal geometry.
 An analogue of  equation \eqref{mainequ-1general} with $\varrho=1$ is also called a $(n-1)$ type fully nonlinear equation.
     In recent years, $(n-1)$ type Monge-Amp\`ere equation 
    has been studied extensively in complex geometry, due to its  close 
    connection to the Calabi-Yau theorem for   Gauduchon and balanced metrics
    \cite{FuWangWuFormtype2010,GTW15}.  

First, we recall the concept of admissible and maximal solution.
 \begin{definition}
 	 	    We say $u$ is an \textit{admissible} function for the equation \eqref{mainequ-1general}  if
 	 $$\lambda(g^{-1}(\Delta u \cdot g-\varrho\nabla^2u+A(x,\nabla u)))\in \Gamma$$  pointwisely.  	 
 	 Meanwhile, we say $u$ is a \textit{maximal} solution of an equation, 
 	  if
 	 $u\geqslant w$ in $M$ for any admissible solution $w$ to the same equation. 
 	  
 	 Similarly, we have  notions of admissible and maximal conformal metrics.
 \end{definition}

Below, we solve the Dirichlet problem with infinite boundary value condition.  
  
  \begin{theorem}
  	\label{thm1-pde-general}
  	Let $(M,g)$ be a compact connected Riemannian manifold  of dimension $n\geqslant 3$  with smooth boundary.  
  	Suppose,  in addition to 
 \eqref{concave}, \eqref{assumption-4},  \eqref{psi-gamma1}, \eqref{sup-infty}, 	\eqref{homogeneous-1-buchong2} and 	\eqref{R-gamma1},
   that there is a 
  	$C^2$-admissible function on $\bar M$. Then the equation \eqref{mainequ-1general} admits a smooth admissible solution $u$
  	with $\underset{x\to \partial M}{\lim}\,u(x)=+\infty.$ 
  	
  	
  \end{theorem}

When the background manifold is complete and noncompact, 
we can solve the equation, 
given an asymptotic condition at infinity which is indeed sharp
\begin{equation}
	\label{asymptotic-condition1-general}
	\begin{aligned}
		f(\lambda(g^{-1}(\Delta \underline{u}\cdot g-\varrho \nabla^2 \underline{u}+A(x,\nabla \underline{u})))) \geqslant  \psi(x,\underline{u}), 
		\mbox{ }
		\forall x\in M.
	\end{aligned}
\end{equation} 
\begin{theorem}
	\label{thm2-pde-general}
	Assume $(M,g)$ is a complete noncompact Riemannian manifold of dimension $n\geqslant 3$ and suppose a $C^2$-admissible function satisfying 
  \eqref{asymptotic-condition1-general}. 
  In addition to \eqref{assumption-4},
  \eqref{psi-gamma1} and \eqref{R-gamma1}, we assume $(f,\Gamma)$ satisfies \eqref{concave}, 	\eqref{sup-infty} and	\eqref{homogeneous-1-buchong2}. 
	Then the equation \eqref{mainequ-1general} 
	possesses
	a unique smooth maximal admissible solution $u$.
	Moreover, $u\geqslant \underline{u}$ in $M$. 	
	
\end{theorem}

To prove the above results, 
the primary problem is to 
derive local estimates for first and second derivatives.
The key ingredient is that, in the presence of \eqref{concave}, \eqref{assumption-4}, \eqref{sup-infty} and \eqref{homogeneous-1-buchong2}, we can apply
Theorem \ref{yuan-k+1}  to confirm the 
\textit{fully uniform ellipticity} of 
equation \eqref{mainequ-1general} in Proposition \ref{key-lemma1}.
This is the most important step  for our approach.
With this crucial procedure 
at hand, we can prove the above theorems. To be brief,  
   we apply singularity solutions of  a class of semi-linear  equations obtained by McOwen \cite{McOwen1993} to construct supersolutions. Then we   complete the proof of Theorem \ref{thm1-pde-general}. 
In an attempt 
to show  Theorem \ref{thm2-pde-general}, 
we
set up appropriate approximate problems on an exhaustion series of domains, building on Theorem \ref{thm1-pde-general}.
 The approximation is based on
   a  maximum principle: 
 the approximate Dirichlet problems with infinite boundary value condition produce
 a decreasing sequence of approximate
  solutions (see 
 Proposition \ref{thm-c0-upper-2}). 
 

    \subsection{Prescribed problems for complete conformal metrics} 
  Let  $\mathrm{Sec}_g$, ${Ric}_g$ and 
  ${R}_g$ denote the sectional, Ricci and scalar curvature of $g$, respectively. 
  When $n\geqslant3$ we denote for $g$
  \begin{equation}
  	\begin{aligned}
  		\,& A_{{g}}^{\tau,\alpha}=\frac{\alpha}{n-2} \left({Ric}_{g}-\frac{\tau}{2(n-1)}   {R}_{g}\cdot {g}\right), \,& \alpha=\pm1,  \mbox{  }\tau \in \mathbb{R}. \nonumber
  	\end{aligned}
  \end{equation}
When $\alpha=1$, it is the modified Schouten tensor  $A_g^\tau=\frac{1}{n-2} ({Ric}_g-\frac{\tau}{2(n-1)}{R_g}\cdot g ).$
  When $\tau=\alpha=1$, it is the Schouten tensor
  $A_g=\frac{1}{n-2} ({Ric}_g-\frac{1}{2(n-1)}{R_g}\cdot g ).$ 
For $\tau=n-1$ and $\alpha=1$, it corresponds to the Einstein tensor $G_g={Ric}_g-\frac{1}{2}{R}_g\cdot g.$

 One natural question to ask is:
  given a smooth positive 
  function $\psi$, is there 
  a smooth complete metric $\tilde{g}=e^{2{u}}g$ satisfying
  \begin{equation}
  	\label{main-equ1}
  	\begin{aligned}
  		f(\lambda(\tilde{g}^{-1}A_{\tilde{g}}^{\tau,\alpha}))=\psi \mbox{  }  \mbox{ in } M.
  	\end{aligned}
  \end{equation}
  
  The equations of this type
    include many important equations as special cases.
 When $f=\sigma_1$, $\tau=0$ and 
 $\psi$ is a proper constant, 
  it is closely related to the well-known Yamabe problem, proved by Schoen \cite{Schoen1984} 
  combining the important  work  of 
   Aubin \cite{Aubin1976} and Trudinger \cite{Trudinger1968}. 
  For $f=\sigma_k^{1/k}$, $\psi=1$ and $\tau=\alpha=1$,
  the equation on closed Riemannian manifolds   
  was proposed by Viaclovsky \cite{Viaclovsky2000}, known as $k$-Yamabe problem. From then on it has received  much attention  in \cite{Branson2006Gover,ChangGurskyYang2002,Ge2006Wang,Guan2003Wang-CrelleJ,ABLi2003YYLi,Gursky2007Viaclovsky,ShengTrudingerWang2007}.  
When $M$ has boundary,  the prescribed scalar curvature problem with certain boundary properties was studied by \cite{Cherrier1984,Escobar1992JDG,Escobar1992Ann}, and further complemented in \cite{Han1999Li,Brendle2014Chen,Marques2005,Marques2007},
  where the obtained metrics are not 
  complete.
  The case is fairly different when the metric is complete.
  A deep result of Aviles-McOwen \cite{Aviles1988McOwen}, 
  extending the celebrated  theorem of Loewner-Nirenberg \cite{Loewner1974Nirenberg}, 
  yields that every compact Riemannian manifold 
  of dimension $n\geqslant3$ 
  with smooth boundary admits a complete conformal  metric with negative constant scalar curvature. 
  Since then
 Loewner-Nirenberg and Aviles-McOwen's results were extended by many experts to prescribed $\sigma_k$ curvature equation  
  \begin{equation}
  	\label{main-equ-sigmak}
  	\begin{aligned}
  		\sigma_k^{1/k}(\lambda(\tilde{g}^{-1}A_{\tilde{g}}^{\tau,\alpha}))=\psi \mbox{  }  \mbox{ in } M,
  	\end{aligned}
  \end{equation}
 when imposing
  restrictions to $(\alpha,\tau)$
  \begin{equation}
  	\label{tau-alpha-2}
  	\begin{cases}
  		\tau<1 \,& \mbox{ if } \alpha=-1,\\
  		\tau>n-1 \,& \mbox{ if } \alpha=1.
  	\end{cases}
  \end{equation} 
Under this assumption, the equation  is automatically of fully uniform ellipticity according to the formula \eqref{conformal-formula1} below.
   Below we list part of related literature. 
On compact manifolds with boundary, the existence 
of solutions  to prescribed $\sigma_k$-curvature equation 
 for   $-{Ric}_g$ was obtained by Guan  \cite{Guan2008IMRN} and Gursky-Streets-Warren  \cite{Gursky-Streets-Warren2011},  
and the case  
$\tau>n-1$, $\alpha=1$  was considered by  Li-Sheng  \cite{Li2011Sheng}. 
On the other hand,  the  prescribed $\sigma_k$-curvature equation 
on a closed manifold was studied by Gursky-Viaclovsky \cite{Gursky2003Viaclovsky}  with $\tau<1$, $\alpha=-1$,
and by Sheng-Zhang \cite{Sheng2006Zhang} for $\tau>n-1$, $\alpha=1$ based on a flow approach. 
  
  Nevertheless, 
most known results on fully nonlinear   Loewner-Nirenberg problem for modified Schouten tensors require condition \eqref{tau-alpha-2},
 and it is rarely known under a  broader assumption. 
 To close this gap, one challenge we face is the \textit{critical} case $\tau=n-1$, in which  the uniform ellipticity  
  has to break down when $\Gamma=\Gamma_n$.
    This critical case  
  coincides 
  with prescribed curvature equation for the
  Einstein tensor  
  \begin{equation}
  	\label{conformal-equ0}
  	\begin{aligned}
  		f(\lambda(\tilde{g}^{-1}G_{\tilde{g}}))=\psi  \mbox{  }  \mbox{ in } M.
  	\end{aligned}
  \end{equation}
  Unfortunately, the topological obstruction 
  as described in Remark \ref{remark1-crucial} below
  indicates that when $\Gamma=\Gamma_n$, at least for $n=3$, the fully nonlinear Loewner-Nirenberg problem for positive Einstein tensor is unsolvable in general. 
  This 
  manifests that the 
  obstruction 
  to the solvability of 
 \eqref{main-equ1}  
   in the conformal
    class of complete metrics   
   arises primarily from the lack of fully uniform ellipticity. 


  
  \subsubsection{Prescribed problem I: compact manifolds with boundary}  
  \label{PrePro1}

  When restricted to the equation \eqref{main-equ1}, 
  the assumption \eqref{assumption-4}  
  reads as follows
  \begin{equation}
  	\label{tau-alpha}
  	\begin{cases}
  		\tau<1 \,&\mbox{ if } \alpha=-1, \\
  		\tau>1+(n-2)(1-\kappa_\Gamma\vartheta_{\Gamma}) \,&\mbox{ if } \alpha=1.
  	\end{cases}
  \end{equation}
  It would be worthwhile to note that the assumption \eqref{tau-alpha}  is much more broader than \eqref{tau-alpha-2},
  and it allows the critical case $\tau=n-1$
  whenever $\Gamma\neq\Gamma_n$.
  As a corollary of Theorem \ref{thm1-pde-general}, we solve
  \eqref{main-equ1} in the conformal class of complete metrics. 
  \begin{theorem}
  	\label{existence1-compact}
  	In addition to \eqref{concave}, \eqref{homogeneous-1-buchong2} and  
  	\eqref{tau-alpha}, we assume  
  	\begin{equation}
  		\label{homogeneous-1-mu}
  		\begin{aligned}
  			f(t\lambda)=t^{\mathrm{\varsigma}} f(\lambda), \mbox{ } \forall \lambda\in\Gamma, \mbox{} \forall t>0,
  			\mbox{  for some constant } 0<\mathrm{\varsigma} \leqslant1.
  		\end{aligned}
  	\end{equation}
  	Assume that $(M,g)$ is a compact  connected Riemannian manifold of dimension $n\geq 3$ with smooth boundary and support a $C^2$  compact conformal admissible metric.
  	Then for any $0<\psi\in C^\infty(\bar M)$, there exists at least one smooth complete  
  	metric $\tilde{g}=e^{2u}g$  
  	satisfying \eqref{main-equ1}.  
  	  \end{theorem}
  	We construct $C^2$  compact conformal admissible metrics, using 
  	Morse functions. 
  \begin{theorem} 
  	\label{thm1-construction}
  		\label{existence1-compact-construction}
  In Theorem \ref{existence1-compact} the assumption on the existence of a compact conformal admissible metric can be dropped 
  if $(\alpha,\tau)$ further satisfies 
  	\begin{equation}
  		\label{tau-alpha-3}
  		\begin{cases}
  			\tau\leqslant 0  \,& \mbox{ if } \alpha=-1,\\
  			\tau\geqslant 2  \,& \mbox{ if } \alpha=1. 
  		\end{cases}
  	\end{equation}
  	
  \end{theorem}

We apply this construction to draw geometric conclusions. 
For the Ricci tensor case,
together with Theorem \ref{thm1-unique},  Theorem \ref{thm1-construction} deduces the following theorem. 
\begin{theorem}
	\label{thm1-infinitevolume}
Suppose $(M,g)$ is a compact connected Riemannian manifold of dimension $n\geqslant 3$ with smooth boundary. Then there  is   a unique complete conformal metric  $\tilde{g}$ 
with  $Ric_{\tilde{g}}<-1$ and
$\sigma_n/\sigma_{n-1}(\lambda(-Ric_{\tilde{g}}))=1$. 
 \end{theorem}
  This gives a new proof of \cite[Theorem 1]{Lohkamp-2} when $n\geq3$.
 Notice also that the complete metric
  is obtained in each conformal class.
 This is in contrast with \cite[Theorems A and C]{Lohkamp-1} of Lohkamp.  

In addition, we can deform the Einstein tensor.
  \begin{theorem}
  	\label{thm0-conformal}
  	Let $(M,g)$ be a compact connected Riemannian manifold of dimension $n\geqslant3$ with smooth boundary.
 Let $(f,\Gamma)$ satisfy  	\eqref{concave}, \eqref{homogeneous-1-buchong2} and \eqref{homogeneous-1-mu}.
  	Suppose in addition that  $\Gamma\neq \Gamma_n$. 
  	Then for each $0<\psi\in C^\infty(\bar M)$,  there exists a smooth admissible complete metric 
  	$\tilde{g}=e^{2u}g$ satisfying \eqref{conformal-equ0}.

  \end{theorem}

  Note that this theorem is significantly interesting 
  in dimension three, 
  since the Einstein tensor is closely connected 
  with the sectional curvature.
  Based on this, 
   	we present a topological obstruction to  reveal 
    	  that the condition $\Gamma\neq\Gamma_n$ imposed in 
  Theorem \ref{thm0-conformal} 
  is crucial and cannot be further dropped. 
  \begin{remark}[Topological obstruction]
  	\label{remark1-crucial}
  Denote $B_{r}(a)=\left\{x\in \mathbb{R}^3: |x-a|^2<r^2\right\}$. Let 
  	$\bar B_{r_1}(a_1), \cdots,\bar B_{r_m}(a_m)$ be pairwise disjoint.
  	For $r\gg1$ with $\cup_{i=1}^m B_{r_i}(a_i)\subset B_r(0)$, we denote
  	$\Omega=B_{r+1}(0)\setminus (\cup_{i=1}^m \bar B_{r_i}(a_i))$,  $g$ a Riemannian metric on $\Omega$.
  	Fix $x\in \Omega$,  let $\Sigma\subset T_x\Omega$ be a tangent $2$-plane,   $\vec{\bf n}\in T_x\Omega$ the unit normal vector to $\Sigma$, then
  	\begin{equation} \label{GSW0} \begin{aligned}
  			G_g(\vec{\bf n},\vec{\bf n})=-\mathrm{Sec}_g(\Sigma), \nonumber
  	\end{aligned}  \end{equation}
  	see \cite[Section 2]{Gursky-Streets-Warren2010} or \eqref{GSW1} below.
  	If Theorem \ref{thm0-conformal} holds in the positive cone case,
  	then the solution on $\Omega$ is
  a complete conformal
  	metric  
  	with negative sectional curvature.
  	This contradicts to the Cartan-Hadamard theorem.  
  	
  	
  \end{remark}

  \subsubsection{Prescribed problem II: complete noncompact manifolds}
  \label{PrePro2}

  In contrast with 
  the resolution of Yamabe problem on closed Riemannian manifolds, 
  the 
  complete noncompact version of
  Yamabe problem 
  is not always solvable as shown by
  Jin \cite{Jin1988}. 
  Consequently, 
  one could only expect  
  the solvability of prescribed curvature problem in the conformal class of complete noncompact metrics
  under proper additional assumptions. 
  Indeed prior to Jin's work, Ni \cite{Ni-82Invent,Ni-82Indiana} obtained the existence and nonexistence of solutions to prescribed scalar curvature equation on standard 
  Euclidean spaces, which was partially extended by Sui \cite{Sui2017JGA} to the prescribed $\sigma_k$-curvature equation for negative Ricci curvature.  
  When imposed fairly strong restrictions to the asymptotic ratio of prescribed functions and
  curvature, or even to the topology  
  of background manifolds, 
  Aviles-McOwen \cite{Aviles1985McOwen,Aviles1988McOwen2}  and Jin \cite{Jin1993} investigated prescribed scalar curvature equation on negatively curved complete noncompact Riemannian manifolds, 
followed by 
  Fu-Sheng-Yuan \cite{FuShengYuan}   recently, who studied 
   prescribed $\sigma_k$-curvature equation 
 \eqref{main-equ-sigmak} for $\alpha=-1$ and $\tau<1$.
Unfortunately, the restrictions imposed there
  are 
  not optimal, even for  conformal  prescribed  scalar curvature equation.
  
  It still remains widely open to determine under which conditions 
 the conformal prescribed curvature equation is solvable on complete noncompact manifolds.
  We 
  close this gap
as a consequence of Theorem \ref{thm2-pde-general}. That is, we can solve 
fully nonlinear version of Yamabe problem 
 on a complete noncompact Riemannian manifold, 
 assuming 
 existence of
  a $C^2$ complete admissible 
  metric $\underline{g}=e^{2\underline{u}}g$ 
  with
  \begin{equation}
  	\label{key-assum1}
  	\begin{aligned}
  		{f(\lambda(\underline{g}^{-1}A_{\underline{g}}^{\tau,\alpha}))}   \geqslant \Lambda_0 {\psi} \mbox{ holds uniformly in $M\setminus K_0$.}
  	\end{aligned}
  \end{equation}
  Here $\Lambda_0$ is a uniform positive constant, and $K_0$ is a compact subset of $M$.

  \begin{theorem}
  	\label{thm1}
  	Let $(M,g)$ be a complete noncompact Riemannian manifold of dimension $n\geqslant 3$. 
  	Suppose $(f,\Gamma)$ satisfies   \eqref{concave}, \eqref{homogeneous-1-buchong2} and \eqref{homogeneous-1-mu}. 
  	Given a smooth positive function $\psi$ and $(\alpha,\tau)$ obeying \eqref{tau-alpha}, we assume that
  	$(M,g)$ carries a  $C^2$ complete admissible conformal metric subject to \eqref{key-assum1}.
  	Then there  
  	exists 
  	a unique smooth complete maximal conformal  admissible metric 
  	satisfying 
  	 \eqref{main-equ1}. 
  \end{theorem}

  \begin{remark}
  	Theorems \ref{thm2-pde-general} and \ref{thm1} 
  	reveal that all of geometric and analytic obstructions to  
  	solvability 
  	of the corresponding equations are 
  	embodied in the assumption of asymptotic condition at infinity, which is sufficient and necessary:
  	\begin{itemize}
  		\item  The solution we expect to obtain tautologically satisfies the asymptotic condition at infinity, which is therefore necessary for the existence. 
  		\item  
  		Prescribed curvature problem \eqref{main-equ1} 
  		is not always solvable. Please refer to  \cite{Jin1993,Ni-82Indiana,Sui2017JGA} for some nonexistence results on conformal prescribed scalar and Ricci curvature equations. 
  		
  	\end{itemize}
  	In conclusion, 
  	except the asymptotic assumption 
  	at infinity, it is not required to impose further restrictions to curvatures, prescribed functions and the topology 
  	 of background manifolds. This is in contrast to related works cited above.
  	This is new even in
  	Euclidean spaces. 
 An analogue of Theorem \ref{thm1} also holds for conformal prescribed scalar curvature equation,  see Theorem \ref{thm-scalarcurvature} below. 
  \end{remark}

  
  

  

    \begin{remark}	A somewhat surprising fact to us is	that we impose neither \eqref{elliptic} nor \eqref{elliptic-weak} in  Theorems \ref{yuan-k+1},  \ref{thm1-pde-general}, \ref{thm2-pde-general},  \ref{existence1-compact}, \ref{existence1-compact-construction}, \ref{thm0-conformal} and \ref{thm1}.	This is in contrast with huge literature on fully nonlinear elliptic equations.  \end{remark}

In conclusion, in this paper we first prove partial uniform ellipticity for fully nonlinear equations and then use it to confirm the uniform ellipticity of a special class of fully nonlinear
 equations.  As an application, we solve a fully nonlinear Loewner-Nirenberg problem, given a compact conformal admissible metric. This assumption is partially confirmed, utilizing Morse theory. 
As a result, we can deduce various geometric conclusions.
  Furthermore, we   examine asymptotic behavior and uniqueness of solutions to fully nonlinear Loewner-Nirenberg problem. 
Building on these solutions,
 we prove the existence and uniqueness of complete maximal conformal metric solving  a complete noncompact version of fully nonlinear Yamabe problem, under an asymptotic condition at infinity which is in effect sharp.  
 Our approach also works for conformal scalar curvature equation.
 Moreover, the topological obstruction 
 indicates that
  our strategy is delicate. 
 \vspace{1mm}
 
  The paper is organized as follows.  
  In Section \ref{section3} we investigate the partial uniform ellipticity.  
  This is one of the most important parts of the paper.
  In  Section \ref{section7} we   confirm  the uniform ellipticity of \eqref{mainequ-1general}. This is crucial for our approach.
  In Section \ref{Sec3} we 
  solve a class of equations of fully uniform ellipticity on complete noncompact Riemannian manifolds, 
using an approximate argument.   
 The key ingredient is 
  the solution of the Dirichlet problem with infinite boundary value condition. 
  The proof of  existence of such solutions 
  is left to Section \ref{DP-1}. 
  Moreover,   in Section \ref{section6} 
  the approximation method is also used to 
  analyze asymptotic behavior
  and  uniqueness of complete conformal metrics.  
  In Section \ref{sec16} we construct conformal admissible metrics, based on a result on Morse functions.
 In the final section, we derive local and boundary estimates for equations of fully uniform ellipticity. 
  


    \medskip
  
  \section{The partial uniform ellipticity}
  \label{section3}

  We start with the following lemma.
  \begin{lemma} 
  	\label{lemma3.4} 
  	For $(f,\Gamma)$ satisfying 
  	\eqref{concave}, the following 
  	statements	are equivalent.
  	\begin{enumerate}
  		\item[$(1)$] $(f,\Gamma)$ satisfies \eqref{addistruc}.
  		\item[$(2)$] For each $\lambda$, $\mu\in \Gamma,$ $\sum_{i=1}^n f_i(\lambda)\mu_i>0.$
  		\item[$(3)$] $f(\lambda+\mu)>f(\lambda)$,   $\forall\lambda$, $\mu\in\Gamma$.
  	\end{enumerate}
  \end{lemma}
  
  \begin{proof}
  	
  	
  	It follows from the concavity of $f$ that
  	\begin{equation}
  		\label{concavity1}
  		\begin{aligned}
  			f(\lambda)\geqslant f(\mu)+\sum_{i=1}^n f_i(\lambda)(\lambda_i-\mu_i), \quad \forall \lambda, \mbox{ }\mu\in \Gamma.
  		\end{aligned}
  	\end{equation}
  	

  	$\mathrm{(1)}\Rightarrow \mathrm{(2)}$: 
  	Fix $\lambda\in \Gamma$. The condition \eqref{addistruc} implies that for any  
  	$\mu \in \Gamma$, there is $T\geqslant1$ (may depend on $\mu$) such that for each $t>T$,
  	$f(t\mu)>f(\lambda)$. Together with \eqref{concavity1},
  	one gets $\sum_{i=1}^n f_i(\lambda) (t\mu_i-\lambda_i)>0$. Thus, $\sum_{i=1}^nf_i(\lambda)\lambda_i>0$ (if one takes $\mu=\lambda$) and moreover
  	$\sum_{i=1}^n f_i(\lambda)\mu_i>0$.
  	
  	$\mathrm{(2)}\Rightarrow \mathrm{(3)}$:
  	The proof uses \eqref{concavity1}.
  	
  	
  	$\mathrm{(3)}\Rightarrow \mathrm{(1)}$: 
  	For any $\lambda,\mbox{ } \mu\in\Gamma$, 
  	$t\lambda-\mu\in\Gamma$ for $t> t_{\lambda,\mu}$, depending only on $\lambda$ and $\mu$.
  	So $f(t\lambda)>f(\mu)$ for such $t$.
  \end{proof}

\begin{remark}
	In the presence of 
	\eqref{concave} and \eqref{elliptic}, 
	this lemma was initially proposed by the author in 
	\cite[Lemma 3.2]{yuan-regular-DP}
	to set up quantitative boundary estimate of the form   
	\begin{equation}	\label{quantitative-boundary-estimate}		\begin{aligned}			\sup_{\partial M}|\partial\overline{\partial}u| \leqslant C(1+\sup_{M}|\partial u|^2)  \nonumber  		\end{aligned} 	\end{equation}
	for the Dirichlet problem of fully nonlinear elliptic equations on complex manifolds.
	In Lemma \ref{lemma3.4}, the assumption \eqref{elliptic} has been removed. 
\end{remark}

  \begin{corollary}
  	\label{coro3.2}
  	If $f$ satisfies 
  	\eqref{concave} and \eqref{addistruc}, then \eqref{elliptic-weak} 
  	holds.
  	
  	
  \end{corollary}

   \begin{proof}	Fix $\lambda \in \Gamma.$ 	According to Lemma \ref{lemma3.4}, $\sum_{i=1}^n f_i(\lambda)\mu_i>0, \mbox{ } \forall\mu\in \Gamma$. 	Note that $\Gamma_n\subseteq\Gamma$,  	then $f_i(\lambda)\geqslant 0,$ $\forall 1\leqslant i\leqslant n.$	We have  $\sum_{i=1}^n f_i(\lambda)>0$ by setting $\mu=\vec{\bf 1}$. \end{proof}

In the following proposition, we relate $\kappa_\Gamma$ to the maximal count of negative components of vectors in $\Gamma$. 
 Let's denote 
  \begin{definition}

  \begin{equation}\begin{aligned}	\widetilde{\kappa}_{\Gamma}=\max\left\{k: (-\alpha_1,\cdots,-\alpha_k,\alpha_{k+1},\cdots, \alpha_n)\in \Gamma, \mbox{ where } \alpha_j>0, \mbox{ } \forall 1\leqslant j\leqslant n \right\}. \nonumber  \end{aligned}\end{equation}
    	 
\end{definition}
  
  \begin{proposition}
  	\label{yuan-kappa-2}
  	$\kappa_\Gamma=\widetilde{\kappa}_\Gamma$.
  \end{proposition}
  \begin{proof}
  	The case $\Gamma=\Gamma_n$ is true since $\kappa_{\Gamma_n}=0$ and $\widetilde{\kappa}_{\Gamma_n}=0$.
  	Next, we consider the case $\Gamma\neq\Gamma_n$.
  	Assume  $(-\alpha_1,\cdots,-\alpha_{\widetilde{\kappa}_\Gamma}, \alpha_{\widetilde{\kappa}_\Gamma+1},\cdots,\alpha_n)\in \Gamma$, for $\alpha_i>0$, $\forall 1\leqslant i\leqslant n$. Then 
  	$(0,\cdots,0, \alpha_{\widetilde{\kappa}_\Gamma+1},\cdots,\alpha_n)\in \Gamma$, 
  	which implies 
  	 $ {\kappa}_\Gamma\geqslant \widetilde{\kappa}_\Gamma$.
  	Conversely, if 
  	$$({\overbrace{0,\cdots,0}^{{\kappa}_\Gamma}},{\overbrace{1,\cdots, 1}^{n-{\kappa}_\Gamma}})\in \Gamma,$$
  	  then 
  	  by the openness of $\Gamma$, we have for some $0<\epsilon\ll1$,
  	$$({\overbrace{-\epsilon,\cdots,-\epsilon}^{{\kappa}_\Gamma}},{\overbrace{1,\cdots, 1}^{n-{\kappa}_\Gamma}})\in \Gamma.$$   
  	Consequently, $ \widetilde{\kappa}_\Gamma\geqslant  {\kappa}_\Gamma$.
  \end{proof}

      Corollary \ref{coro3.2} is a part of Theorem \ref{yuan-k+1}.
  Below we complete the proof. 
   Fix $\lambda\in\Gamma$. 
  It follows from  the concavity and symmetry of $f$ that 
  \begin{equation}
  	f_i(\lambda)\geqslant f_j(\lambda) \mbox{ for } \lambda_i\leqslant\lambda_j.  \nonumber
    \end{equation}
  In particular
  \begin{equation}
  	\label{1nf_i}
  	f_1(\lambda)\geqslant \frac{1}{n}\sum_{i=1}^n f_i(\lambda) \mbox{ if } \lambda_1\leqslant \cdots\leqslant\lambda_n. \nonumber
  \end{equation}
  Therefore, for the case $\Gamma=\Gamma_n$ $(\mbox{i.e. } \kappa_\Gamma=0)$, we immediately obtain Theorem \ref{yuan-k+1}.
  For general $\Gamma$, 
  Theorem \ref{yuan-k+1} is a consequence of Proposition \ref{yuan-kappa-2} and   the following proposition.
   \begin{proposition}
  	\label{yuanrr-2}
  	Assume $\Gamma\neq \Gamma_n$ and  $f$ satisfies 
  	\eqref{concave} and \eqref{addistruc} in $\Gamma$. 
  	For the $\kappa_\Gamma$ as defined in Definition \ref{yuan-kappa}, let
  	$\alpha_1, \cdots, \alpha_n$ be $n$ strictly positive constants such that 
  	$$(-\alpha_1,\cdots,-\alpha_{\kappa_\Gamma}, \alpha_{\kappa_\Gamma+1},\cdots, \alpha_n)\in \Gamma.$$
  	In addition,  assume $\alpha_1\geqslant\cdots\geqslant \alpha_{\kappa_\Gamma}$.
  	Then for each $ \lambda\in \Gamma$ with order $\lambda_1 \leqslant \cdots \leqslant\lambda_n$,
  	\begin{equation}
  		\label{theta1}
  		\begin{aligned}
  			f_{\kappa_\Gamma+1}(\lambda)\geqslant\frac{\alpha_1}{\sum_{i=\kappa_\Gamma+1}^n \alpha_i-\sum_{i=2}^{\kappa_\Gamma}\alpha_i}f_1(\lambda).
  		\end{aligned}
  	\end{equation}
  \end{proposition}
  
  \begin{proof}
  	Fix $\lambda\in\Gamma$ and we assume
  	$\lambda_1\leqslant \lambda_2\leqslant \cdots \leqslant\lambda_n$. 
  	 From Corollary \ref{coro3.2}, $f_i(\lambda)\geqslant0$ and $\sum_{i=1}^n f_i(\lambda)>0$.
  	By Lemma \ref{lemma3.4},
  	\begin{equation}
  		\label{good1-yuan}
  		\begin{aligned}
  			-\sum_{i=1}^{\kappa_\Gamma} \alpha_i f_i(\lambda)+\sum_{i=\kappa_\Gamma+1}^n \alpha_i f_i(\lambda)>0
  		\end{aligned}
  	\end{equation}
which simply yields $f_{\kappa_\Gamma+1}(\lambda)>   \frac{\alpha_1}{\sum_{i=\kappa_\Gamma+1}^n \alpha_i}f_1(\lambda)$.
  	In addition, 
  	one can derive \eqref{theta1} by using 
  	the iteration.
  	
  \end{proof}

  \begin{remark}
  	\label{remark111}
  	When $\Gamma\neq\Gamma_n$ the constant $\vartheta_\Gamma$ in Theorem \ref{yuan-k+1} can be achieved as follows:
  	$$\vartheta_\Gamma= \sup_{(-\alpha_1,\cdots,-\alpha_{\kappa_\Gamma}, \alpha_{\kappa_\Gamma+1},\cdots, \alpha_n)\in \Gamma,\, \alpha_i>0} \, \frac{\alpha_1/n}{\sum_{i=\kappa_\Gamma+1}^n \alpha_i-\sum_{i=2}^{\kappa_\Gamma}\alpha_i}.$$	
  	
  	A somewhat remarkable fact to us is that 
  	in the description of partial uniform ellipticity, 
  	$\kappa_\Gamma$ and $\vartheta_\Gamma$	
  	depend  only on $\Gamma$ instead specifically on $f$.  	
  	Consequently, it has great advantages in applications to
  	PDEs and geometry.   For instance, Theorem \ref{yuan-k+1} allows us to 
  	compute the partial uniform ellipticity of  
  	fully 
  	nonlinear equations in  connection with the notions of
  	$p$-convexity \cite{ShaJP1986,WuHH-1987},  $\mathbb{F}$-subharmonic and $\mathbb{G}$-plurisubharmonic functions in the sense of  Harvey-Lawson \cite{Harvey2011Lawson}. 
  \end{remark}

  The following proposition precisely reveals  the 
   interaction between $\kappa_\Gamma$ and partial uniform ellipticity. Denote
  \begin{equation}
  	\label{Proj-Rk}
  	\Gamma^{\infty}_{\mathbb{R}^k}:=  \left\{\lambda'\in\mathbb{R}^k: (\lambda',c,\cdots,c)\in\Gamma \mbox{ for some } c>0\right\}.
  \end{equation}
  \begin{proposition}
  	\label{thm-k+1}
  	
  	In addition to 
  	\eqref{concave}, 
  	we assume  for some $1\leqslant k\leqslant n-1$  that
  	$f$ is of $(k+1)$-uniform ellipticity in 
  	 $\Gamma$.  Then $\kappa_\Gamma\geqslant k$ and
  	$\Gamma^{\infty}_{\mathbb{R}^k}=\mathbb{R}^k.$
  	
  \end{proposition}

  \begin{proof}

  	
  	  	Let $\vartheta$ be as in \eqref{partial-uniform2}. 
  	As above, $\vec{\bf 1}=(1,\cdots,1)$.
  	For the $(f,\Gamma)$, we denote    
  	\begin{equation}
  		\begin{aligned}
  			\Gamma^\sigma=  \{\lambda\in\Gamma: f(\lambda)>\sigma\}.  \nonumber
  	\end{aligned} \end{equation}
  	Let   
  	$c_0$ be the positive constant with
  	$f(c_0 \vec{\bf 1})>\sup_{\partial\Gamma}f$,
  	and then we set $a=1+c_0$. Thus $f(a\vec{\bf 1})>f(c_0 \vec{\bf1})$ by $\sum_{i=1}^n f_i(\lambda)>0$.
  	For $0<\epsilon<a$ and $R=\frac{a}{\vartheta}$, we denote
  	$\lambda_{\epsilon,R}=({\overbrace{\epsilon,\cdots,\epsilon}^{k}},{\overbrace{R,\cdots, R}^{n-k}}).$
  	We can deduce from \eqref{concavity1} that
  	\begin{equation}
  		\begin{aligned}
  			f(\lambda_{\epsilon,R})\geqslant
  			\,& f(a\vec{\bf 1})+\epsilon\sum_{i=1}^{k} f_i(\lambda_{\epsilon,R})+
  			R\sum_{i=k+1}^n  f_i(\lambda_{\epsilon,R})
  			-a\sum_{i=1}^n f_i(\lambda_{\epsilon,R}) 
  			\\
  			 \geqslant
  		 \,& f(a\vec{\bf 1})+(R\vartheta -a)\sum_{i=1}^{n} f_i(\lambda_{\epsilon,R}) \mbox{ (using $(k+1)$-uniform ellipticity)}
  	\\ 	= \,& f(a\vec{\bf 1}) \mbox{ (noticing } R=\frac{a}{\vartheta})
  			\\>\,&  f(c_0\vec{\bf1}). \nonumber
  		\end{aligned}
  	\end{equation}
  	So $\lambda_{\epsilon,R}=(\epsilon,\cdots,\epsilon, R,\cdots, R)\in \overline{\Gamma^{f(a\vec{\bf 1})}}$.
  	Notice that $R=\frac{a}{\vartheta}$ does not depend  on $\epsilon$. 
  	Taking $\epsilon\rightarrow 0^+$,  we get $(0,\cdots,0,R,\cdots,R)\in  \overline{\Gamma^{f(a\vec{\bf 1})}}
  	\subset\Gamma^{f(c_0\vec{\bf1})}\subset\Gamma.$
Thus  	$\kappa_\Gamma\geqslant k$. 
  \end{proof}
  
  \begin{remark}
  	The proof uses the simple fact $\Gamma_n\subseteq\Gamma$. This is different from the proof of Proposition \ref{construction-type2} below.
  \end{remark}
  
  \begin{corollary}
  	\label{thm-sharp}
  	The $(\kappa_\Gamma+1)$-uniform ellipticity asserted in Theorem \ref{yuan-k+1} 
  	cannot be improved.
  \end{corollary}
  
  The following two corollaries are key ingredients in the construction of local barriers in Section \ref{DP-1}.
  \begin{corollary}
  	\label{coro-type2}
  	In addition to  \eqref{concave}, 
  	we assume that $f$ is of fully uniform ellipticity in $\Gamma$. Then the corresponding cone $\Gamma$ is of type 2.
  \end{corollary}

  \begin{corollary}
  	\label{thm-type2}
  	If $\Gamma$ carries a smooth symmetric concave $f$ satisfying \eqref{concave} 
  	and \eqref{addistruc},
  	then the following statements are equivalent to each other.
  	\begin{enumerate}
  			\item[$\mathrm{(1)}$]  $f$ is of fully uniform ellipticity in $\Gamma$.
  		\item[$\mathrm{(2)}$] $\Gamma$ is of type  2. That is $\Gamma^{\infty}_{\mathbb{R}^{n-1}}=\mathbb{R}^{n-1}$, where $\Gamma^{\infty}_{\mathbb{R}^{n-1}}$ is as we denoted in \eqref{Proj-Rk}.
  	
  	\end{enumerate}
  \end{corollary}

  \begin{lemma}
  	\label{lemma-add1}
  	Let $\kappa_\Gamma$ be 
  	as in Definition \ref{yuan-kappa},  $\vartheta_\Gamma$ be  as 
  	in Theorem \ref{yuan-k+1}, then 
  	\begin{enumerate}
  		\item[$\mathrm{(1)}$] $\kappa_\Gamma$ is an integer with $0\leqslant \kappa_\Gamma\leqslant n-1$.  
  		\item[$\mathrm{(2)}$]  For $\Gamma=\Gamma_k$, $\kappa_{\Gamma}=n-k$.
  		\item[$\mathrm{(3)}$]  $\kappa_\Gamma=0$ if and only if $\Gamma= \Gamma_n$.
  		\item[$\mathrm{(4)}$] $\kappa_\Gamma=n-1$  if and only if $\Gamma$ is of type 2.
  		
  		\item[$\mathrm{(5)}$] $0<\vartheta_{\Gamma}\leqslant \frac{1}{n}$.
  	\end{enumerate}
  	Moreover, if $\vartheta_{\Gamma}=\frac{1}{n}$ and $\kappa_\Gamma=n-1$ occur simultaneity then $$f_i(\lambda)=\frac{1}{n}\sum_{j=1}^n f_j(\lambda), \quad\forall \lambda\in\Gamma, \mbox{ } \forall 1\leqslant i\leqslant n. $$
  \end{lemma} 
  \begin{proof}
  	
  	Note that $\Gamma\subseteq\Gamma_1$, we know $\kappa_\Gamma\leqslant n-1$.
  	Next we prove the last statement. Set $\lambda=\vec{\bf1}$, then  
  	$f_i(\vec{\bf1})=\frac{1}{n}\sum_{j=1}^n f_j(\vec{\bf1}), \mbox{ } \forall 1\leqslant i\leqslant n$. This
  	simply yields $\vartheta_\Gamma\leqslant \frac{1}{n}$. 
  	The proofs of the other three statements are obvious, and we omit them here.
  \end{proof}

  As consequences of Lemma \ref{lemma-add1}, we obtain the following lemmas.

  \begin{lemma}\label{lemma-add3-2}
  	Suppose $\varrho$ obeys \eqref{assumption-4}. Then $\varrho <n$ and 
  	$\varrho\neq 0. $
  \end{lemma}
  

  \begin{lemma}\label{lemma-add2}
  	Suppose 
  	$(\alpha,\tau)$ satisfies \eqref{tau-alpha}. Then
  	\begin{equation}
  		\label{positive1}
  		\begin{aligned}
  			\alpha(n\tau+2-2n)>0, \quad   (\tau-1)(n\tau+2-2n) >0. \nonumber
  		\end{aligned}
  	\end{equation}

  \end{lemma}

  It seems very interesting
  to answer whether G{\aa}rding's cone 
  is the largest one to share same $\kappa_\Gamma$. 
  
  \begin{problem}
  	\label{problem1}
  	Fix $1< k<n$.  For the cone $\Gamma$ with $\kappa_\Gamma=n-k$, is $\Gamma\subseteq\Gamma_{k}$ true?
  	
  \end{problem}
  
  \begin{problem}
  	\label{problem2}
  	For any $\Gamma$, $\hat{\Gamma}$ with $\kappa_\Gamma\geqslant \kappa_{\hat{\Gamma}}+1$, is $\hat{\Gamma}\subset\Gamma$ true?
  \end{problem}

  Finally, we confirm  \eqref{addistruc} under  the assumptions \eqref{sup-infty} and \eqref{homogeneous-1-buchong2}.
  
  \begin{lemma}
  	\label{lemma2.3}
  	
  	Assume,  in addition to \eqref{concave}, that 
  	$\sup_\Gamma f=+\infty$ and
  	\begin{equation}
  		\label{addistruc-0}
  		\begin{aligned}
  			\lim_{t\rightarrow+\infty} f(t\lambda)>-\infty, \mbox{  } \forall\lambda\in\Gamma.
  		\end{aligned}
  	\end{equation}
  	Then $(f,\Gamma)$ satisfies \eqref{addistruc}.
  \end{lemma}
  \begin{proof}
  	We have by \eqref{concavity1} the following inequality 
  $$\sum_{i=1}^n f_i(\lambda)\mu_i \geqslant \limsup_{t\rightarrow+\infty} f(t\mu)/t, \quad  \forall \lambda, \mbox{  } \mu\in\Gamma.$$ 
Thus  
  	\begin{equation}	\label{addistruc-1}	\begin{aligned}		\sum_{i=1}^n f_i(\lambda)\mu_i\geqslant0,  \quad 	\forall \lambda, \mbox{  } \mu\in\Gamma.  	\end{aligned}	\end{equation}
  	So $f_i(\lambda)\geqslant 0$, $\forall \lambda\in \Gamma, 1\leqslant i\leqslant n$.  Together with $\sup_\Gamma f=+\infty$ and 
  	\eqref{concavity1}, we derive 
  	\begin{equation}
  		\label{sum>0}
  		\begin{aligned}
  			\sum_{i=1}^n f_i(\lambda)>0 \mbox{ in } \Gamma.
  		\end{aligned}
  	\end{equation} 
  	By \eqref{addistruc-1}, \eqref{sum>0} and the openness of $\Gamma$, 
  	we get  $(2)$ from Lemma \ref{lemma3.4}, as required.
  	
  \end{proof}

   \medskip
  
  \section{Problem reduction}
  \label{section7}

  \subsection{Preliminaries}
  \label{preli-0}

  
  On a Riemannian manifold $(M,g)$, one defines the curvature tensor
  by
  $$R(X,Y)Z=-\nabla_X\nabla_YZ+\nabla_Y\nabla_X Z+\nabla_{[X,Y]}Z.$$

  Let $e_1,...,e_n$ be a local frame on $M$.  From now on we denote  
  $$  \langle X,Y\rangle=g(X,Y), \quad g_{ij}= \langle e_i,e_j\rangle, \quad  \{g^{ij} \} =  \{g_{ij} \}^{-1}.$$
  Under Levi-Civita connection  of $(M,g)$, $\nabla_{e_i}e_j=\Gamma_{ij}^k e_k$, and $\Gamma_{ij}^k$ denote the Christoffel symbols.  
  The curvature coefficients are given by
  $R_{ijkl}=\langle e_i,R(e_k,e_l)e_j\rangle$. 
  For simplicity we write 
  $$\nabla_i=\nabla_{e_i}, \nabla_{ij}=\nabla_i\nabla_j-\Gamma_{ij}^k\nabla_k, 
  \nabla_{ijk}=\nabla_i\nabla_{jk}-\Gamma_{ij}^l\nabla_{lk}-\Gamma^l_{ik}\nabla_{jl}, \mbox{ etc}.$$
  
  Under the local unit orthogonal frame 
  \begin{equation}  \begin{aligned}
  		\,& {Ric}_g(e_i,e_i)=\sum_{j=1}^n \langle R(e_i,e_j)e_i,e_j \rangle,
  		\,& {R}_g=\sum_{i=1}^n {Ric}_g(e_i,e_i), 
  \end{aligned}  \end{equation}
  \begin{equation} \label{GSW1} \begin{aligned}
  		{Ric}_g(e_i,e_i)=
  		\frac{{R}_g}{2}-\frac{1}{2}\sum_{k,l\neq i} \langle R(e_k,e_l)e_k,e_l\rangle, \mbox{  }  \forall 1\leqslant i\leqslant n.
  \end{aligned}  \end{equation}

  Under the conformal change
  $\tilde{g}=e^{2u}g$, one has  
  (see e.g. \cite{Besse1987})
  \begin{equation}
  	\begin{aligned}
  		{Ric}_{\tilde{g}}=\,& {Ric}_g -\Delta u g -(n-2)\nabla^2u-(n-2)|\nabla u|^2g +(n-2)du\otimes du. \nonumber
  	\end{aligned}
  \end{equation}
  Thus
  \begin{equation}
  	\label{conformal-formula1}
  	\begin{aligned}
  		A_{\tilde{g}}^{\tau,\alpha}
  		= A_{g}^{\tau,\alpha}
  		+\frac{\alpha(\tau-1)}{n-2}\Delta u g-\alpha  \nabla^2 u
  		+\frac{\alpha(\tau-2)}{2}|\nabla u|^2 g
  		+\alpha  du\otimes du. 
  	\end{aligned}
  \end{equation}
  In particular, the Schouten tensor obeys
  \begin{equation}
  	\begin{aligned}
  		\label{conformal-formula2}
  		A_{\tilde{g}} 
  		= A_{g} -\nabla^2 u	-\frac{1}{2}|\nabla u|^2 g
  		+ du\otimes du. 
  	\end{aligned}
  \end{equation}

  \subsection{Fully uniform ellipticity and construction of type 2 cones}

  First, we shall summarize the definition of type 2 cones.
  \begin{definition} 
  	[{\cite{CNS3}}]
  	\label{def-type2}
  	$\Gamma$ is said to be of type 1 if the positive $\lambda_i$ axes belong to $\partial\Gamma$; otherwise it is called of type 2.
  \end{definition}
  
  Except $\Gamma_1$, it seems not easy to find general  type 2 cones. 
   Based on 
  Theorem \ref{yuan-k+1}, we can construct some examples of type 2 cones.
  
  Given a cone $\Gamma$, we take  a constant $\varrho$ satisfying \eqref{assumption-4}.
  Note that  $\varrho<n$ and $\varrho\neq 0$ by Lemma \ref{lemma-add3-2}. We let  
  \begin{equation}	 	\begin{aligned}  		\mu_i=\frac{1}{n-\varrho}
  		\left(\sum_{j=1}^n \lambda_j -\varrho\lambda_i \right), 
  		\mbox{ i.e.  }
  		\lambda_i=\frac{1}{\varrho}\left(\sum_{j=1}^n \mu_j-(n-\varrho)\mu_i\right),	\end{aligned}  \end{equation}
  \begin{equation}
  \label{map1}
  	\begin{aligned}
  		\tilde{\Gamma}  =
  		\left\{(\lambda_1,\cdots,\lambda_n): 
  		\lambda_i=\frac{1}{\varrho}\left(\sum_{j=1}^n \mu_j-(n-\varrho)\mu_i\right),
  		\mbox{ } (\mu_1,\cdots,\mu_n)\in \Gamma 
  		\right\}.
  	\end{aligned}
  \end{equation} 
  So $\tilde{\Gamma}$ is also an open symmetric convex cone of $\mathbb{R}^n$ with  $\tilde{\Gamma}\subseteq\Gamma_1$. 
  One has a symmetric concave function $\tilde{f}$ on  $\tilde{\Gamma}$
  as follows:
  \begin{equation}\label{def-f}\begin{aligned}  \tilde{f}(\lambda)= f(\mu). 
  \end{aligned}\end{equation} 
  
  \begin{proposition} 
  	[Fully uniform ellipticity]
  	\label{key-lemma1}
  	Let $(f,\Gamma)$ satisfy \eqref{concave} and \eqref{addistruc}.  
  	Let  $(\tilde{f},\tilde{\Gamma})$ be as above.
  	Suppose in addition that $\varrho$ satisfies \eqref{assumption-4}. 
  	 Then   there is a uniform positive constant $\theta$ such that
  	\begin{equation}	 
  		\label{FUE-1}
  		\begin{aligned}
  			\frac{\partial \tilde{f}}{\partial \lambda_i}(\lambda) \geqslant \theta \sum_{j=1}^n \frac{\partial \tilde{f}}{\partial \lambda_j}(\lambda)>0 \mbox{ in }  \tilde{\Gamma}, \quad \forall 1\leqslant i\leqslant n.
  		\end{aligned}
  	\end{equation}
  Moreover, $(\tilde{f},\tilde{\Gamma})$ also satisfies \eqref{concave} and \eqref{addistruc}.
  \end{proposition}
  
  \begin{proof}
  	
  	{\bf Case 1}: $\varrho<0$.  
  	A straightforward computation shows
  	$$\frac{\partial \tilde{f}}{\partial\lambda_i}
  	= \sum_{j=1}^n\frac{\partial f}{\partial\mu_j}\frac{\partial\mu_j}{\partial\lambda_i}
  	=\frac{1}{n-\varrho}
  	\left(\sum_{j=1}^n\frac{\partial f}{\partial\mu_j} -\varrho\frac{\partial f}{\partial\mu_i}\right)
  	\geqslant \frac{1}{n-\varrho}\sum_{j=1}^n\frac{\partial f}{\partial\mu_j} =
  	\frac{1}{n-\varrho} \sum_{j=1}^n\frac{\partial \tilde{f}}{\partial\lambda_j}.$$
  	
  	{\bf Case 2}: $0<\varrho<\frac{1}{1-\kappa_\Gamma\vartheta_{\Gamma}}$. 
  	By Theorem \ref{yuan-k+1},  $(1- \kappa_\Gamma\vartheta_{\Gamma} )\sum_{j=1}^n \frac{\partial f}{\partial\mu_j}\geqslant \frac{\partial f}{\partial\mu_i},$
  	$\forall 1\leqslant i\leqslant n$. Thus
  	\begin{equation}
  		\label{computation-1}
  		\begin{aligned}
  			\frac{\partial \tilde{f}}{\partial\lambda_i}
  			\geqslant 
  			\frac{1-\varrho(1-\kappa_\Gamma\vartheta_{\Gamma})}{n-\varrho}\sum_{j=1}^n \frac{\partial f}{\partial\mu_j}
  			=
  			\frac{1-\varrho(1-\kappa_\Gamma\vartheta_{\Gamma})}{n-\varrho} \sum_{j=1}^n\frac{\partial \tilde{f}}{\partial\lambda_j}. \nonumber
  		\end{aligned}
  	\end{equation}
  	
  \end{proof}

  \begin{proposition}
  	[Construction of type 2 cones]
  	\label{construction-type2}
  	Let $\varrho$ satisfy  \eqref{assumption-4}, and let $\tilde{\Gamma}$ be the cone as defined in \eqref{map1}. Then $\Gamma_n\subset\tilde{\Gamma}$ and $\tilde{\Gamma}$ is a type 2 cone.
  \end{proposition}
  \begin{proof}
  	Given $(f,\Gamma)$, as in \eqref{map1} and \eqref{def-f}, 
  	we obtain $(\tilde{f},\tilde{\Gamma})$.
  	It only requires to verify
  	\[(0\cdots,0,1)\in \tilde{\Gamma}.\]
  	
  Denote $\lambda_t=(t\cdots,t,1)$ and 
  $I=\{t\in [0,1]: \lambda_t\in \tilde{\Gamma}\}.$
  	Clearly, $1\in I$, and $I$ is an open subset of $I$ by the openness of $\tilde{\Gamma}$. 
  	It is sufficient to prove the closeness.
  	Let $\{t_i\}\subset I$, $0<t_i\leqslant 1$ and $$t_0=\lim_{i\to +\infty}t_i.$$ 
  	It only requires to prove $t_0\in I$. 
  	We denote  $\tilde{\Gamma}^\sigma=  \{\lambda\in\tilde{\Gamma}: \tilde{f}(\mu)>\sigma\}. $  
  	Let $\theta$ be as in \eqref{FUE-1},    and let   $\tilde{f}(a\vec{\bf1})>\sup_{\partial\tilde{\Gamma}}\tilde{f}$.
  	We deduce from \eqref{FUE-1} and the concavity of $\tilde{f}$ that
  	\begin{equation}
  		\begin{aligned}
  			\tilde{f} \left(\frac{a}{\theta}  \lambda_{t_i}\right)
  	\geqslant	\tilde{f}(a\vec{\bf 1})+\frac{at_i}{\theta} \sum_{j=1}^{n-1} \frac{\partial \tilde{f}}{\partial \lambda_j}	\left(\frac{a}{\theta}\lambda_{t_i}\right)+ \frac{a}{\theta} \cdot\frac{\partial \tilde{f}}{\partial \lambda_n}	\left(\frac{a}{\theta}\lambda_{t_i}\right)-a\sum_{j=1}^n \frac{\partial \tilde{f}}{\partial \lambda_j}\left(\frac{a}{\theta}\lambda_{t_i}\right)  
  			> \tilde{f}(a\vec{\bf 1}). \nonumber
  		\end{aligned}
  	\end{equation}
  	Thus $\frac{a}{\theta} \lambda_{t_0}\in \overline{\tilde{\Gamma}^{\tilde{f}(a\vec{\bf 1})}}$. As required, $\lambda_{t_0}\in\tilde{\Gamma}.$
  	This  gives   $I=[0,1]$.
  \end{proof}

  \begin{corollary}
  	\label{coro3-ingamma}
  	For $\varrho$ satisfying \eqref{assumption-4},  $(1,\cdots,1,1-\varrho)\in \Gamma.$
  \end{corollary}
  
   \begin{remark} From the construction of $\tilde{\Gamma}$, it is not easy to see $\Gamma_n\subset\tilde{\Gamma}$. \end{remark}


  
  \subsection{Reduction to Schouten tensor case}

  We   reduce the prescribed curvature equation \eqref{main-equ1}
  to a uniform elliptic equation for conformal deformation of Schouten tensor. To do this,
  we can check  
  \begin{equation}
  	\label{check-1}
  	\begin{aligned}
  		\mathrm{tr}\left(g^{-1}(-A_g)\right)g
  		-\varrho \left(-A_g\right)
  		=  \frac{n-2}{\alpha(\tau-1)}A_{g}^{\tau,\alpha},
  		\nonumber
  	\end{aligned}
  \end{equation}
where we take
\begin{equation}
	\label{beta-gamma-A}
	\begin{aligned}
		\varrho=\frac{n-2}{\tau-1}.  \nonumber
	\end{aligned}
\end{equation}
  \begin{lemma}\label{lemma-add3}
  	Let $(\alpha,\tau)$ satisfy \eqref{tau-alpha}, let  $\varrho=\frac{n-2}{\tau-1}$. 
  	 Then 
  	\begin{equation}
  		\begin{aligned}
  			\varrho<\frac{1}{1-\kappa_\Gamma \vartheta_{\Gamma}} \mbox{ and } \varrho\neq 0. \nonumber
  		\end{aligned}
  	\end{equation}

  \end{lemma}
  
  Therefore,  when $f$ is  homogeneous of degree $\mathrm{\varsigma}$,  the equation \eqref{main-equ1} is equivalent to 
  \begin{equation}
  	\label{mainequ-02-0-0}
  	\begin{aligned}
  		\tilde{f}(\lambda(-\tilde{g}^{-1}A_{\tilde{g}}))
  		=\left(\frac{n-2}{\alpha(n\tau+2-2n)}\right)^\mathrm{\varsigma} \psi, 
  	\end{aligned}
  \end{equation}
  in which $\tilde{f}$ is of fully uniform ellipticity in $\tilde{\Gamma}$ according to Proposition \ref{key-lemma1}. Here $(\tilde{f},\tilde{\Gamma})$  is as in \eqref{map1}-\eqref{def-f}.
  To study this equation, according to the formula \eqref{conformal-formula2},
  it requires to consider 
  \begin{equation}
  	\label{mainequ-02-0}
  	\begin{aligned}
  		\tilde{f}(\lambda(g^{-1}( \nabla^2 u+\frac{1}{2}|\nabla u|^2 g - du\otimes du-A_g)))
  		=\left( \frac{n-2}{\alpha(n\tau+2-2n)}\right)^\mathrm{\varsigma} \psi e^{2\mathrm{\varsigma} u}.
  	\end{aligned}
  \end{equation} 

  \begin{remark}
  	In what follows, 
  	 we replace $(\tilde{f}, \tilde{\Gamma})$ by $(f,\Gamma)$
  	 if necessary. 
  \end{remark}

    \medskip
  
  \section{The equations on complete noncompact  manifolds and applications to prescribed curvature problem}
  \label{Sec3}

    From Proposition \ref{key-lemma1} and Lemma \ref{lemma2.3},    \eqref{mainequ-1general} falls into the equation of the form 
  \begin{equation}
  	\label{mainequ-1}
  	\begin{aligned}
  		f\left(\lambda(g^{-1}(\nabla^2 u+A(x,\nabla u)))\right) =\psi(x,u), 
  	\end{aligned}
  \end{equation}
  in which 
  ${f}$ is of fully uniform ellipticity in ${\Gamma}$, i.e., 
  there is a uniform constant $\theta$ such that
  \begin{equation}
  	\label{fully-uniform2}
  	\begin{aligned}
  		f_{i}(\lambda)\geqslant \theta\sum_{j=1}^n f_j(\lambda)>0 	\mbox{ in } \Gamma,		\quad \forall 1\leqslant i\leqslant n.
  	\end{aligned}
  \end{equation}
  In addition,
  $\psi(x,t)$ is a smooth positive  function on $M\times \mathbb{R}$ satisfying \eqref{psi-gamma1},
  and $A(x,p)$ is  a smooth symmetric  $(0,2)$-tensor satisfying \eqref{R-gamma1}.

  In this section we solve the equation \eqref{mainequ-1} on a complete noncompact Riemannian manifold, given an asymptotic condition at infinity 
  \begin{equation}
  	\label{asymptotic-condition1}
  	\begin{aligned}
  		f(\lambda(g^{-1}(\nabla^2 \underline{u}+A(x,\nabla \underline{u})))) \geqslant  \psi(x,\underline{u}), 
  		\mbox{ }
  		\forall x\in M.
  	\end{aligned}
  \end{equation} 
  Also, for the equation \eqref{mainequ-1}, we say a $C^2$ function $w$ is an \textit{admissible} function, if 
  \begin{equation}\begin{aligned}		\lambda(g^{-1} (\nabla^2w+A(x,\nabla w)))\in \Gamma.  \nonumber	\end{aligned}\end{equation}

  \begin{theorem}
  	\label{thm2-pde}
  	Assume $(M,g)$ is a complete noncompact Riemannian manifold of dimension $n\geqslant 3$ and suppose a $C^2$-admissible function $\underline{u}$ satisfying 
  	the asymptotic condition \eqref{asymptotic-condition1}.	
  	Suppose in addition that  \eqref{concave},  \eqref{sup-infty},	\eqref{homogeneous-1-buchong2},   \eqref{psi-gamma1}, \eqref{R-gamma1} and  \eqref{fully-uniform2} hold.
  	Then the equation \eqref{mainequ-1} has 
  	a unique smooth maximal admissible solution $u$.
  	Moreover,  $u\geqslant \underline{u}$ in $M.$
  \end{theorem}

  As a consequence, we obtain Theorem \ref{thm2-pde-general}.
  
    \begin{remark} 	Suppose for any   positive constant $\Lambda$ that  	\begin{equation}	\label{psi-asymptotic1}\begin{aligned}\Lambda \psi(x,t)\geqslant \psi(x,t+C_\Lambda), \mbox{ } \forall (x,t)\in M\times\mathbb{R} \nonumber	\end{aligned}\end{equation}  holds	for a constant $C_\Lambda$.  Then the asymptotic condition  \eqref{asymptotic-condition1}  can be replaced by 	\begin{equation}	\label{asymptotic-condition-02}	\begin{aligned}		f(\lambda(g^{-1}( \nabla^2 \underline{u}+A(x,\nabla \underline{u}))))  \geqslant \Lambda \psi(x,\underline{u}).  \nonumber	\end{aligned}  	\end{equation}  \end{remark}
  
  \subsection*{Comparison principle}

  First of all, we deduce a comparison principle. 
  \begin{lemma}
  	\label{lemma-mp}
  	Let	$(\Omega,g)$ be a compact Riemannian manifold with  boundary $\partial\Omega$. 
  	Suppose  \eqref{concave},  
  	 \eqref{sup-infty},	\eqref{homogeneous-1-buchong2} and \eqref{psi-gamma1} 
  	hold.
  	Let $w, v\in C^2(\Omega)\cap C(\bar \Omega)$ be admissible functions subject to
  	\[f(\lambda(g^{-1}(\nabla^2 w+A(x,\nabla w))))\geqslant \psi(x,w), \mbox{ } f(\lambda(g^{-1} (\nabla^2 v+A(x,\nabla v))))\leqslant \psi(x,v)
  	\mbox{ in } \Omega,  \]
  	\[w-v\leqslant 0 \mbox{ on } \partial \Omega.\]
  	Then 
   \begin{equation}	\begin{aligned}	w-v\leqslant 0 \mbox{ in } \Omega.\nonumber \end{aligned} \end{equation}
  \end{lemma}
  \begin{proof} 
 Assume by contradiction that there is an interior point $x_0\in \Omega$ such that $0<(w-v)(x_0)=\sup_{\Omega} (w-v)$. 
  	Therefore at $x_0$,
  	\begin{equation}
  		\begin{aligned}
  			\,& \nabla^2 v \geqslant \nabla^2 w, \,& \nabla v=\nabla w. \nonumber 
  		\end{aligned}
  	\end{equation}
  By Lemma \ref{lemma2.3} and Corollary \ref{coro3.2}, 
 one has \eqref{elliptic-weak}. Thus at $x_0$
  	\begin{equation}
  		\begin{aligned}
  			\,& f(\lambda(g^{-1} (\nabla^2 v+A(x,\nabla v))))\geqslant f(\lambda(g^{-1} (\nabla^2 w+A(x,\nabla w)))), \nonumber 
  		\end{aligned}
  	\end{equation}
  	which further yields $v(x_0)\geqslant w(x_0)$ by \eqref{psi-gamma1}. This contradicts to $w(x_0)>v(x_0)$.
  \end{proof}
  
  \subsection*{Approximation and the proof of Theorem \ref{thm2-pde}}
  
  We employ an approximate method, based on Theorem 
  \ref{thm1-pde} below. 
  Let $\{M_k\}_{k=1}^{+\infty}$ be an exhaustion series of domains  
  with
  \begin{itemize}
  	\item $M=\cup_{k=1}^{\infty} M_k, \mbox{  }  \bar M_k =M_k\cup\partial M_k,\mbox{  }  \bar M_k\subset\subset M_{k+1}$.
  	
  	\item  $\bar M_k$ is a compact $n$-manifold with smooth boundary $\partial M_k$.
  \end{itemize}
 
  According to Theorem 
  \ref{thm1-pde}, for each $k\geqslant1$,  there is an admissible function $u_k\in C^\infty(M_k)$  satisfying
  \begin{equation}
  	\label{approximate-DP2-2}
  	\begin{aligned}
  		f\left(\lambda(g^{-1}(\nabla^2 u_k+A(x,\nabla u_k)))\right) =\psi(x,u_k)   \mbox{ in } M_k,
  	\end{aligned}
  \end{equation}
  \begin{equation}
  	\label{approximate-DP3-2}
  	\begin{aligned}
  		\lim_{x\rightarrow\partial M_k} u_k(x)=+\infty.  
  	\end{aligned}
  \end{equation}

  By the comparison principle (Lemma \ref{lemma-mp}), we immediately derive
  \begin{proposition}
  	\label{thm-c0-upper-2}
  	Let $u_k$ be the admissible solution to problem \eqref{approximate-DP2-2}-\eqref{approximate-DP3-2}, then
  	\begin{equation}
  		\begin{aligned}
  			u_{k+1}\leqslant u_k \mbox{ in } M_k, \quad \forall k\geqslant1. \nonumber
  		\end{aligned}
  	\end{equation} 
  \end{proposition}


  \begin{proposition}
  	\label{thm-c0-lower-2}
  	For any admissible solution $u_k$ to \eqref{approximate-DP2-2}-\eqref{approximate-DP3-2}, we have 
  	\begin{equation}
  		u_k\geqslant \underline{u} \mbox{ in } M_k, \quad \forall k\geqslant1. \nonumber
  	\end{equation} 
  \end{proposition}


Proposition \ref{thm-c0-upper-2} states that $\{u_k\}_{k=m}^\infty$ is a decreasing sequence  of functions on $M_m$; while such a decreasing  sequence 
is uniformly bounded from below according to Proposition \ref{thm-c0-lower-2}.
Thus we can take the limit 
  \begin{equation}
	\label{uinfty-lim}
	u_\infty=\lim_{k\to+\infty} u_k.
\end{equation} 
Moreover,
$u_\infty\geqslant \underline{u}$  in $ M.$
 In Propositions \ref{thm-c0-upper-2}-\ref{thm-c0-lower-2}, and Theorems \ref{interior-2nd-2}-\ref{thm-gradient2}  below, we establish  the local estimates up to second order derivatives  for each $u_k$.
Combining with
  Evans-Krylov theorem \cite{Evans82,Krylov83} and classical Schauder theory, 
 we know that $u_\infty$ is in fact a smooth admissible solution to the equation \eqref{mainequ-1}.


  
 Let $w_\infty $ be another admissible solution to the equation \eqref{mainequ-1}, then by the comparison principle we have 
 $u_k\geqslant w_\infty$ in $M_k$ for all $k\geqslant1$. Thus
  	$u_\infty(x)\geqslant w_\infty(x),$ $\forall x\in M.$
  	This means that the solution given by \eqref{uinfty-lim}  is the maximal solution.  The uniqueness of maximal solution is obvious.

  \subsection*{Prescribed curvature problem on complete noncompact  manifolds}
 \label{solving-Schouten1}
  
 Note that the equation  \eqref{mainequ-02-0} is a special case of \eqref{mainequ-1}.  
  As a consequence of Theorem \ref{thm2-pde}, we obtain 
  the following result and Theorem \ref{thm1}.
  \begin{theorem}
  	\label{thm1-shouten}
  	Suppose $(f,\Gamma)$  satisfies \eqref{concave}, \eqref{homogeneous-1-buchong2}, \eqref{homogeneous-1-mu} and 
  	\eqref{fully-uniform2}. Let $(M,g)$ be a complete noncompact Riemannian manifold of dimension $n\geqslant 3$ and with a $C^2$   complete  conformal  metric $\underline{g}$ subject to 
  	\begin{equation}
  		\label{admissible-metric-schouten1}
  		\begin{aligned}
  			\lambda(-g^{-1}A_{\underline{g}})\in \Gamma \mbox{ } \mbox{ in } M  
  		\end{aligned}
  	\end{equation}
  and
  	\begin{equation}
  		\label{key-assum2-2}
  		\begin{aligned}
  	 	f(\lambda(-\underline{g}^{-1} A_{\underline{g}})) \geqslant \Lambda_1  \psi \mbox{ } \mbox{ in } M  
  		\end{aligned}
  	\end{equation}   
  	where $0<\psi\in C^\infty(M)$ and $\Lambda_1$ is a uniform positive constant.
  	Then there  
  	exists 
  	a unique 
  	smooth   maximal   complete  metric $\tilde{g}=e^{2u}g$ satisfying 
  	$f(\lambda(-\tilde{g}^{-1} A_{\tilde{g}})) =\psi$ and $f(\lambda(- {g}^{-1} A_{\tilde{g}})) \in\Gamma.$
  \end{theorem}

  \begin{remark}
  	
  	When 
  	$M \subset\mathbb{R}^n$  is a bounded domain, the  continuous
  	viscosity solution for $f(\lambda(-\tilde{g}^{-1}A_{\tilde{g}}))=1$ was studied by 
  	Gonz\'alez-Li-Nguyen \cite{GLN-2018} using Perron's method.   	The Dirichlet problem for $\sigma_k^{1/k}(\tilde{g}^{-1}A_{\tilde{g}})=\psi$ was studied by Guan \cite{Guan2007AJM}, while the obtained metric is not complete.
  	
  	Our method is fairly different, which is based on partial uniform ellipticity. 
  	In addition, the metrics obtained in this paper are smooth and complete. 
  \end{remark}

  As a consequence of Theorem \ref{thm1}, 
  together with the Maclaurin inequalities
  \begin{equation}
  	\begin{aligned} \left({\sigma_k(\lambda)}/{C_n^k}\right)^{1/k}
  		\leqslant \left({\sigma_l(\lambda)}/{C_n^l}\right)^{1/l}, 
  		\quad \forall\lambda\in \Gamma_k, 
  		\mbox{ }   1\leqslant l< k\leqslant n  \nonumber
  	\end{aligned}
  \end{equation}
  we obtain the existence of 
  solutions to prescribed $\sigma_k$-curvature equations.
  \begin{theorem}
  	\label{thm2-A_gtau}
  	Fix $\Gamma=\Gamma_k$ and $(\alpha,\tau)$ satisfying \eqref{tau-alpha}, we 
  	assume $(M,g)$ is a   complete noncompact Riemannian manifold of dimension $n\geqslant3$ 
  	with $\lambda(g^{-1}A_g^{\tau,\alpha})\in\Gamma_k$   and
  	\begin{equation}  \begin{aligned}
  			\sigma_k^{1/k}(\lambda(g^{-1}A_g^{\tau,\alpha}))\geqslant \delta \nonumber
  	\end{aligned}  \end{equation}
  for some positive constant $\delta.$
  	Then
  	there is a unique   $(\tilde{g}_1, \tilde{g}_2, \cdots, \tilde{g}_k)$, which consists of 
  	smooth maximal complete conformal metrics,
  	to satisfy
  	\[\begin{cases}
  		-{R}_{\tilde{g}_1}=\frac{n(n-2)}{\alpha(n\tau+2-2n)}, \\
  		\sigma_2(\lambda(\tilde{g}_2^{-1}A^{\tau,\alpha}_{\tilde{g}_2}))=\frac{n(n-1)}{2},   \mbox{ } \lambda({g}^{-1}A^{\tau,\alpha}_{\tilde{g}_2})\in \Gamma_2,
  		\\
  		\quad\quad\quad\quad \vdots \\
  		\sigma_k(\lambda(\tilde{g}_k^{-1}A^{\tau,\alpha}_{\tilde{g}_k}))=\frac{n!}{k!(n-k)!},   \mbox{ } \lambda({g}^{-1}A^{\tau,\alpha}_{\tilde{g}_k})\in \Gamma_k.
  	\end{cases}\]
  	
  \end{theorem}
  
  \vspace{1mm}
  The approximate method also works for conformal scalar curvature equation.  
  \begin{theorem}	\label{thm-scalarcurvature}
  	Let $(M,g)$ be a complete noncompact Riemannian manifold of dimension $n\geqslant3$ with  negative scalar curvature.
  	Let $\psi$ be a smooth positive function which is bounded from above in terms of $-{R}_g$, 
  	i.e. there is a positive constant $\delta$ such that 
  	$0<\psi\leqslant -\delta {R}_g.$  Then there exists a unique  smooth maximal complete conformal metric $\tilde{g}$ with prescribed scalar curvature $-{R}_{\tilde{g}}=\psi.$
  \end{theorem}

  \begin{remark}
  	One can apply Theorem \ref{thm-scalarcurvature} to improve \cite[Theorem A]{Aviles1988McOwen2} of Aviles-McOwen. More precisely, 
  	let $(M,g)$ be a complete noncompact Riemannian manifold with nonpositive scalar curvature subject to ${R}_g (x)< -\delta < 0$ 
  	outside a compact subset $K$,
  	then for any 
  	smooth positive function  $\psi$  with $\sup_M\psi<+\infty$,
 there is a unique smooth maximal complete conformal metric $\tilde{g}$ with $-R_{\tilde{g}}=\psi$.
  	
  \end{remark}
  
 
  \medskip
  \section{The Dirichlet problem}
  \label{DP-1}
  
  As described in Section \ref{Sec3},  
  it only requires to solve the equation \eqref{approximate-DP2-2} on  $M_k$
  with infinite boundary value condition \eqref{approximate-DP3-2}.
  Throughout this section, we assume for simplicity that $(M,g)$ is a compact connected Riemannian manifold with smooth boundary $\partial M$, and $\varphi\in C^\infty(\partial M)$.
  
  \subsection{The Dirichlet problem with finite boundary value condition}

  First we solve the Dirichlet problem with finite boundary value condition.  
  \begin{theorem}
  	\label{thm1-finiteBVC}
  	Suppose, in addition to \eqref{concave}, \eqref{sup-infty}, \eqref{homogeneous-1-buchong2},   \eqref{psi-gamma1}, \eqref{R-gamma1} and \eqref{fully-uniform2},  that there is a $C^2$ admissible function.  Then  the equation \eqref{mainequ-1} admits a unique smooth admissible solution $u$  with $u|_{\partial M }=\varphi$.
  \end{theorem}
  
  \subsubsection*{$C^0$-estimate}
  
  By the maximum principle, we obtain    $C^0$-estimate. 
  \begin{lemma}
  	\label{lemma-c0general}
  	In addition to \eqref{concave}, 
  	\eqref{elliptic-weak},
  	\eqref{psi-gamma1}, 
  	\eqref{sup-infty} and \eqref{homogeneous-1-buchong2},
  	we assume that there is an admissible function $\underline{w}$. Let  
  	$u\in C^2(\bar M)$ be an admissible solution to the equation \eqref{mainequ-1} with
  	$u=\varphi$  on $\partial  M$,
  	then 
  \begin{equation}
  	\begin{aligned}
  		\min\left\{\inf_{\partial M}(\varphi-\underline{w}), \mbox{ } A_1-\sup_M\underline{w}\right\}\leqslant
  		 u-\underline{w} \leqslant 		\max\left\{\sup_{\partial M}(\varphi-\underline{w}), \mbox{ } A_2-\inf_M\underline{w} \right\},
  		\nonumber
  	\end{aligned}
  \end{equation}
  	where  
  	\begin{equation}
  		\begin{aligned}
  			\,&	\sup_{x\in\bar M}\psi(x,A_1) \leqslant \inf_{x\in \bar M}   f(\lambda(g^{-1}(\nabla^2\underline{w}+A(x,\nabla\underline{w})))),  \\
  			\,&	 \inf_{x\in\bar M}\psi(x,A_2) \geqslant \sup_{x\in \bar M} f(\lambda(g^{-1}(\nabla^2\underline{w}+A(x,\nabla\underline{w})))).
  			\nonumber
  		\end{aligned}
  	\end{equation}
  	
  \end{lemma}
  
  \subsubsection*{Boundary gradient estimate}
 We denote
  the distance from $x$ to $\partial M$ with respect to $g$ by  
  \begin{equation}
  	\label{distance-function}
  	\mathrm{\sigma}(x)=\mathrm{dist}_g(x,\partial M).
  \end{equation}
  Then $\mathrm{\sigma}(x)$ is smooth near the boundary.
  We use  $\mathrm{\sigma}$ to construct local barriers. 
Denote 
  \begin{equation}
  	\label{omega-delta}
  	\Omega_\delta:=\{x\in M: \mathrm{\sigma}(x)<\delta\}.
  \end{equation} 

The assumption \eqref{R-gamma1} on $A(x,p)$ yields that 
 there exists a positive uniform constant  $B$ such that
\begin{equation}
	\label{yuan-418}
	\begin{aligned} 
	\big|	\mathrm{tr}(g^{-1}A(x, p))\big| \leqslant  B(1+|p|^2). 
	\end{aligned}
\end{equation}
  
    \noindent{\bf Local upper barrier}.
  Take $\bar{w}=\epsilon \log(1+\frac{\mathrm{\sigma}}{\delta^2})+\varphi,$ where
  $\epsilon$ is a positive constant to be determined later. 
  By \eqref{yuan-418}, one can see that there is a uniform constant $C_B'$ such that
  \begin{equation}
  	\begin{aligned}
  		\mathrm{tr}(g^{-1}A(x,\bar w)) \leqslant C_B' \left(
  		1+|\nabla\varphi|^2+\frac{\epsilon^2|\nabla\mathrm{\sigma}|^2}{(\delta^2+\mathrm{\sigma})^2}\right).  \nonumber
  	\end{aligned}
  \end{equation}
  
  Set $\epsilon=\frac{1}{2C_B'}$.
  The computation 
  shows that   on $\Omega_\delta$, 
  \begin{equation}
  	\begin{aligned}
  		\mathrm{tr}(g^{-1} A(x,\nabla\bar{w}))
  		\,&	
  		+\Delta \bar{w}
  	  		\leq
  		 \frac{ 2(\delta^2+\mathrm{\sigma})\Delta \mathrm{\sigma} -|\nabla\mathrm{\sigma}|^2}{4C_B'(\delta^2+\mathrm{\sigma})^2}
  		 +C_B'|\nabla \varphi|^2
  		 +C_B'+\Delta\varphi
  	\leqslant 0 \nonumber  
  	\end{aligned}
  \end{equation}
provided $0<\delta\ll1$.
  Notice also that 
  $\underset{\delta\to 0^+}{\lim}\, \bar{w}|_{\mathrm{\sigma}=\delta} =+\infty.$
Combining with the $C^0$-estimate in   Lemma \ref{lemma-c0general},  when   $\delta$ is sufficiently small, we know
  \begin{equation} 	\label{upper-barrier1}
  	\begin{aligned}
  		u\leqslant \bar{w} \mbox{ on } \Omega_{\delta}.  
  	\end{aligned}
  \end{equation}

  \noindent{\bf Local lower barrier}.
  Fix $k\geqslant 1$. Pick  $\varepsilon$ a positive constant to be determined later.
  Inspired by \cite{Guan2008IMRN} we take
  \begin{equation}
  	\label{h-def-k}
  	\begin{aligned}
  		w=\varepsilon\log \frac{\delta^2}{k\mathrm{\sigma}+\delta^2}+\varphi.   
  	\end{aligned}
  \end{equation}
  The straightforward computation gives that for any $k\geqslant1$,
  \begin{equation}
  	\label{compute-h-c1c2}
  	\begin{aligned}
  		\nabla w=-\frac{k\varepsilon \nabla \mathrm{\sigma}}{\delta^2+k\mathrm{\sigma}}+\nabla\varphi,\quad
  		\nabla^2 w= \frac{k^2\varepsilon d\mathrm{\sigma}\otimes d\mathrm{\sigma}  }{(\delta^2+k\mathrm{\sigma})^2}-\frac{k\varepsilon\nabla^2 \mathrm{\sigma}}{\delta^2+k\mathrm{\sigma}}+\nabla^2\varphi.   \nonumber
  	\end{aligned}
  \end{equation}
  Note that $|\nabla\mathrm{\sigma}|=1$ on $\partial M$.
  By 
  \eqref{yuan-418} again,  there exist positive constants $C_B$ and $\delta'$ such that
  \begin{equation}
  	\begin{aligned}
  		\left|\nabla^2\varphi +A\left(x, \nabla\varphi-\frac{k\varepsilon \nabla\mathrm{\sigma}}{\delta^2+k\mathrm{\sigma}} \right)\right|  \leqslant \frac{k^2\varepsilon^2C_B |\nabla \mathrm{\sigma}|^2}{(\delta^2+k\mathrm{\sigma})^2} \nonumber
  		 \mbox{ in } \Omega_\delta, \mbox{ } \forall 0<\delta<\delta'.
  	\end{aligned}
  \end{equation}

    By Lemma \ref{lemma2.3}, $f$ satisfies
  \eqref{addistruc}.  
  According to Corollary \ref{coro-type2} or \ref{thm-type2},
 $$\Gamma^{\infty}_{\mathbb{R}^{n-1}}=\mathbb{R}^{n-1}
 \mbox{ and }
  (0,\cdots,0,1)\in\Gamma.$$
    Notice   $\lambda( g^{-1}\frac{\varepsilon k^2d\mathrm{\sigma}\otimes d\mathrm{\sigma}  }{(\delta^2+k\mathrm{\sigma})^2})=\frac{\varepsilon k^2|\nabla\mathrm{\sigma}|^2}{(\delta^2+k\mathrm{\sigma})^2}(0,\cdots,0,1)$. 
  We  take $\varepsilon$, $\delta$ small enough   
  such that
  \begin{equation}
  	\begin{aligned}
  		\lambda\left(g^{-1}\left(   \frac{\varepsilon k^2 d\mathrm{\sigma}\otimes d\mathrm{\sigma}  }{2(\delta^2+k\mathrm{\sigma})^2}
  		-\frac{ \varepsilon k \nabla^2 \mathrm{\sigma}}{\delta^2+k\mathrm{\sigma}}
  		+\nabla^2\varphi 
  		+A\left(x, \nabla\varphi-\frac{ \varepsilon k \nabla\mathrm{\sigma}}{\delta^2+k\mathrm{\sigma}} \right) \right)\right) \in\Gamma \mbox{ in } \Omega_\delta.
  		\nonumber
  	\end{aligned}
  \end{equation}
Fix  $\varepsilon$, $\delta$ as above. From Lemma \ref{lemma3.4}, 
 for $0<\delta\ll1$,   we have on $\Omega_\delta$
  \begin{equation}
  	\begin{aligned}
  		f\left(\lambda(g^{-1}(\nabla^2w+A(x,\nabla w)))\right) \geqslant
  		f\left( \frac{\varepsilon  k^2 |\nabla\mathrm{\sigma}|^2}{2(\delta^2+k\mathrm{\sigma})^2}
  		\left( 0,\cdots,0,1\right)\right). \nonumber
  	\end{aligned}
  \end{equation}
  Since $\psi_z(x,z)>0$ and $w\leqslant\varphi$,
  we further obtain for $0<\delta\ll1$
  \begin{equation}
  	\begin{aligned}
  		f\left( \frac{k^2\varepsilon  |\nabla\mathrm{\sigma}|^2}{2(\delta^2+k\mathrm{\sigma})^2}
  		\left( 0,\cdots,0,1\right)\right)
  		\geqslant \psi(x,w) \mbox{ on } \Omega_\delta. \nonumber
  	\end{aligned}
  \end{equation}
  It follows from Lemmas \ref{lemma-mp}, \ref{lemma-c0general} and 
  $\underset{\delta\to 0^+}{\lim}\,	w|_{\mathrm{\sigma}=\delta} =-\infty$
  that 
  \begin{equation}
  	\label{low-barrier1}
  	\begin{aligned}
  		u\geqslant w=  \varepsilon\log \frac{\delta^2}{k\mathrm{\sigma}+\delta^2}+\varphi 
  		 \mbox{ on } \Omega_{\delta}   
  	\end{aligned}
  \end{equation}  
for some positive constant $\delta$.

 From \eqref{low-barrier1} (setting $k=1$) and \eqref{upper-barrier1} we derive
  \begin{equation}
  	\begin{aligned}
  		|\nabla u|\leqslant C \mbox{ on } \partial M. \nonumber
  	\end{aligned}
  \end{equation}


  Combining with local estimates in Theorems \ref{interior-2nd-2}-\ref{thm-gradient2}, and boundary estimate in Theorem \ref{thm2-bdy}, we obtain estimates up to second order
  \begin{equation}
  	\begin{aligned}
  		|\nabla^2 u| \leqslant C. \nonumber
  	\end{aligned}
  \end{equation} 
  Together with Evans-Krylov theorem and Schauder theory, one can derive higher estimates. By the standard continuity method, we obtain Theorem \ref{thm1-finiteBVC}.


  
  \subsection{The Dirichlet problem with infinite boundary value condition}
  \label{Sec4}

  \begin{theorem}
  	\label{thm1-pde}
  	Let $(M,g)$ be a compact Riemannian manifold of dimension $n\geq 3$ with smooth boundary. 
  	Suppose in addition that
  	\eqref{concave},  \eqref{psi-gamma1}, \eqref{R-gamma1},
  	\eqref{sup-infty},	\eqref{homogeneous-1-buchong2}  and	\eqref{fully-uniform2} hold. Then the equation \eqref{mainequ-1} admits a smooth admissible solution $u$
  	with  
  	$\underset{x\to \partial M}{\lim}\, u(x)=+\infty,$ provided that there is a $C^2$ admissible function on $\bar M$.
  	
  	
  \end{theorem}

  The following   lemma is   useful. It is a consequence of the concavity.
  \begin{lemma}\label{lemma-add5}
  	Let $f$ satisfy \eqref{concave} and $f(\vec{\bf1})<\sup_\Gamma f$, then
  	\begin{equation}
  		\label{key2-main}
  		\begin{aligned}
  			\sum_{i=1}^n \lambda_i\geqslant
  			n+ A_f \left(f(\lambda)-f(\vec{\bf 1})\right), \mbox{ }\forall \lambda\in\Gamma, 
  		\end{aligned}
  	\end{equation}
  	where   $ A_f= {n}\left(\sum_{i=1}^n f_i(\vec{\bf 1})\right)^{-1}.$ 
  	In particular, if $f$ satisfies in addition 
  	  \begin{equation}
  		\label{homogeneous-1}
  		\begin{aligned}
  			f(t\lambda)=tf(\lambda),   
  			\mbox{   } \forall \lambda\in\Gamma, \mbox{  } t>0,
  			\mbox{ with normalization } f(\vec{\bf 1})=1,
  		\end{aligned}
  	\end{equation}
  	then
  	\begin{equation}
  		\label{key1-main}
  		\begin{aligned}
  			\sum_{i=1}^n \lambda_i \geqslant nf(\lambda), \mbox{ }\forall \lambda\in\Gamma.
  		\end{aligned}
  	\end{equation}
  	
  \end{lemma}

  \begin{proof}[Proof of Theorem \ref{thm1-pde}]
  	Let $u^{(k)}$ be an admissible solution to  
  	\begin{equation}
  		\label{mainequ-k}
  		\begin{aligned}
  			\,& f\left(\lambda(g^{-1}(\nabla^2 u^{(k)}+A(x,\nabla u^{(k)})))\right) =\psi(x,u^{(k)}), \,&
  			u^{(k)}|_{\partial M}=\varepsilon\log k,
  		\end{aligned}
  	\end{equation}
  	where the constant 
  	$\varepsilon$ 
  	comes from \eqref{h-def-k}.
  	The comparison principle (Lemma \ref{lemma-mp}) simply yields
  	\begin{equation}
  		\label{c0-lower}
  		\begin{aligned}
  			u^{(k+1)} \geqslant u^{(k)}, \quad \forall k\geqslant 1.
  		\end{aligned}
  	\end{equation}
Namely, $\{u^{(k)}\}$ is an increasing sequence.
  	From Lemma \ref{lemma-add5} we get
  	\begin{equation}
  		\begin{aligned}
  			A_f \psi(x,u^{(k)})\leqslant \Delta  u^{(k)}+  \mathrm{tr}(g^{-1}A(x,\nabla u^{(k)}))+ A_f f(\vec{\bf 1})-n. \nonumber
  		\end{aligned}
  	\end{equation}
  	On the other hand, from \eqref{yuan-418}  
  	\begin{equation}
  		\label{yuan-419}
  		\begin{aligned}
  			\mathrm{tr}(g^{-1}A(x,\nabla u^{(k)}))\leqslant  B(1+|\nabla u^{(k)}|^2). 
  		\end{aligned}
  	\end{equation}
  	
  	By the assumption \eqref{psi-gamma1} on $\psi(x,t)$ and the compactness of $\bar M$, 
  	there are   constants $\Lambda_0$ and $\Theta$ such that
  	\begin{equation}
  		\label{yuan-417}
  		\begin{aligned}
  			\psi(x,u^{(k)})\geqslant \Lambda_0 e^{\Theta u^{(k)}}. \nonumber
  		\end{aligned}
  	\end{equation}
  	Fix $\Theta$. 	
  	Take $q$ a constant bigger than 1 but  sufficiently close to $1$, such that
  	\begin{equation}
  		\label{yuan-420}
  		\begin{aligned}
  			\frac{\Theta}{q-1}\geqslant B, \nonumber
  		\end{aligned}
  	\end{equation} 
  where $B$ is as in \eqref{yuan-419}.
  	Fix such $q$.
  	As a corollary of \cite[Theorem 2.2]{McOwen1993}, 
  	there is $\tilde{u}\in C^\infty(M)$ 
  	such that
  	\begin{equation}
  		\begin{aligned}
  			\Delta  \tilde{u}+\,& \frac{\Theta}{q-1}  |\nabla \tilde{u}|^2
  			 =\frac{q-1}{\Theta}  e^{\Theta \tilde{u}} 
  			\mbox{ }
  			\mbox{ in } M,  \,&	
  			\lim_{x\to\partial M} \tilde{u}(x)\to+\infty.
  		\end{aligned}
  	\end{equation}

  	We know that $u^{(k)}-\tilde{u}$ achieves its maximum at an interior point $x_0\in M$. At $x_0$,
  	\begin{equation}
  		\begin{aligned}
  			\nabla u^{(k)}=\nabla\tilde{u}, \quad \nabla^2 u^{(k)} \leqslant \nabla^2\tilde{u}. \nonumber
  		\end{aligned}
  	\end{equation}
  	Thus
  	\begin{equation}
  		\begin{aligned}
  			A_f\Lambda_0 e^{\Theta (u^{(k)}-\tilde{u})} \leqslant \frac{q-1}{\Theta} 
  			 + \left( A_f f(\vec{\bf 1})+B \right)e^{-\Theta \tilde{u}}.  \nonumber
  		\end{aligned}
  	\end{equation}
  	Notice also that $ \underset{M}{\inf}\, \tilde{u}>-\infty.$
  	We conclude
  	\begin{equation}
  		\label{c0-upper}
  		\begin{aligned}
  			u^{(k)}\leqslant \tilde{u}+C
  		\end{aligned}
  	\end{equation}
  	for a uniform positive constant $C$, independent of $k$. 
  	
  	We have derived the (local) $C^0$-estimate in \eqref{c0-lower} and  	\eqref{c0-upper}. We then have 
  	\begin{equation}
  		\label{uinfty}
  		\begin{aligned}
  			u(x)=\lim_{k\to+\infty} u^{(k)}(x), \mbox{  } \forall x\in M. 
  		\end{aligned}
  	\end{equation}
  	
  Combining with Evans-Krylov theorem, 
 Schauder theory as well as
  the local estimates up to second  derivatives established in 
  Theorems
  	\ref{interior-2nd-2}-\ref{thm-gradient2} below, we can conclude that $u$ is  a smooth admissible solution to \eqref{mainequ-1}  
  	with 
  	$\underset{x\to \partial M}{\lim}\, u(x)=+\infty.$
  	
  	Moreover, the inequality \eqref{low-barrier1} gives (note $\varphi=\varepsilon\log k$)
  	\begin{equation}
  		\label{low-barrier-k}
  		\begin{aligned}
  			u^{(k)}\geqslant   \varepsilon\log \frac{k\delta^2}{k\mathrm{\sigma}+\delta^2}  \mbox{ on } \Omega_{\delta}, 
  		\end{aligned}
  	\end{equation} 
  	which implies  $u  \geqslant -\varepsilon\log\mathrm{\sigma}-C_0$  for some $C_0$ near the boundary.

  \end{proof}

  \begin{proposition}
  	\label{prop-7.5}
  	Suppose that $A(x,du)$ in the equation \eqref{mainequ-1} is of the
  	form
  	\begin{equation}
  		\label{A-form1}
  		\begin{aligned}
  			A(x,\nabla u) =\mathrm{U}+\alpha(x)|\nabla u|^2 g+\beta(x) du\otimes du+ R(x,du)
  		\end{aligned}
  	\end{equation} 
  	where $\mathrm{U}$ is a smooth symmetric $(0,2)$ tensor, 
  $R(x,p)$ is a smooth symmetric $(0,2)$ 
  	tensor depending linearly on $p$. Assume in addition that 
  	$\alpha(x)$ and $\beta(x)$  are smooth functions with
  	\begin{equation}
  		\label{alpha-beta-form1}
  		\begin{aligned}
  			(\alpha(x),\cdots,\alpha(x),\alpha(x)+\beta(x)+1) \in \Gamma \mbox{ in } \bar M.
  		\end{aligned}
  	\end{equation} 
  	Then in  \eqref{h-def-k} we can choose $\varepsilon=1$, hence the solution $u$ (defined by \eqref{uinfty})
  	in Theorem \ref{thm1-pde}  
  	satisfies		
  	\begin{equation}
  		\label{u-asymptotic-2}    
  		\begin{aligned}
  			u \geqslant - \log\mathrm{\sigma}-C_0		
  		\end{aligned}
  	\end{equation}
  	for some $C_0$ near the boundary. Moreover, the metric $\tilde{g}=e^{2u}g$ is complete.
  	
  \end{proposition}

  \begin{proof}
  	We take  $\varepsilon=1$ in \eqref{h-def-k}, i.e. 	$w= \log \frac{k\delta^2}{k\mathrm{\sigma}+\delta^2}$. 
  	By \eqref{compute-h-c1c2}, 
  	\begin{equation}
  		\begin{aligned}
  			\nabla^2 w+A(x,\nabla w)=  U+
  			\frac{k^2\alpha|\nabla\mathrm{\sigma}|^2g+k^2(1+\beta)  d\mathrm{\sigma}\otimes d\mathrm{\sigma}}{(\delta^2+k\mathrm{\sigma})^2}	
  			 -\frac{k \nabla^2 \mathrm{\sigma}}{\delta^2+k\mathrm{\sigma}}  
  			+R\left(x, \frac{-kd \mathrm{\sigma}}{\delta^2+k\mathrm{\sigma}}\right)  \nonumber
  		\end{aligned}
  	\end{equation}  
  	Note
  	that \begin{equation}
  		\begin{aligned}
  			\lambda\left(g^{-1}\left( 	\frac{k^2\alpha|\nabla\mathrm{\sigma}|^2g+k^2(1+\beta)  d\mathrm{\sigma}\otimes d\mathrm{\sigma}}{(\delta^2+k\mathrm{\sigma})^2}\right)\right) 
  			=\frac{k^2 |\nabla\mathrm{\sigma}|^2}{(\delta^2+k\mathrm{\sigma})^2}(\alpha,\cdots,\alpha,\alpha+\beta+1). \nonumber
  		\end{aligned}
  	\end{equation}  
  	So we can prove
  	by  Lemma \ref{lemma3.4}, 
  	that for $0<\delta\ll1$, 
  	\begin{equation}
  		\begin{aligned}
  			f\left(\lambda(g^{-1}(\nabla^2w+A(x,\nabla w)))\right)  
  			\geqslant \psi(x,w) \mbox{ on } \Omega_\delta. \nonumber
  		\end{aligned}
  	\end{equation}
  	
  	Similarly, we obtain an inequality analogous to \eqref{low-barrier1}  (or \eqref{low-barrier-k}) near  boundary 
  	\begin{equation}
  		\begin{aligned}
  			u^{(k)}\geqslant   \log \frac{k\delta^2}{k\mathrm{\sigma}+\delta^2} 
  			\nonumber
  		\end{aligned}
  	\end{equation} 
  	as required.
  \end{proof}
  
  \begin{remark}
  	Since $\Gamma$ is of type 2 (by Corollary \ref{thm-type2} or \ref{coro-type2}), the assumption \eqref{alpha-beta-form1} is automatically satisfied, provided 
  	\begin{equation}
  		\begin{aligned}
  			\alpha(x)\geqslant0, \mbox{ } \alpha(x)+\beta(x)+1\geqslant 0, \mbox{ } 
  			n\alpha(x)+\beta(x)+1>0.  \nonumber
  		\end{aligned}
  	\end{equation} 
  	
  \end{remark}

  As a corollary, we get
  the following result and Theorem \ref{existence1-compact}.
  \begin{theorem}
  	\label{existence1-compact-2}
  	Suppose $(f,\Gamma)$  satisfies \eqref{concave}, \eqref{homogeneous-1-buchong2}, \eqref{homogeneous-1-mu}  and 
  	\eqref{fully-uniform2}.
  	Let $(M,g)$ be a compact connected Riemannian manifold of dimension $n\geqslant 3$ with smooth boundary and support a $C^2$  
  	conformal metric satisfying \eqref{admissible-metric-schouten1}.
  	Then for any $0<\psi\in C^\infty(\bar M)$, there exists at least one smooth  
  	complete  
  	metric $\tilde{g}=e^{2u}g$ 
  	satisfying 
  	${f}(\lambda(-\tilde{g}^{-1}A_{\tilde{g}}))= \psi$
  	and  $\lambda(-{g}^{-1}A_{\tilde{g}})\in \Gamma$
  	in $M$.
  	
  \end{theorem}

    \medskip

  \section{Construction of admissible metrics}
  \label{sec16}

  Now we focus on the equation \eqref{main-equ1}. Throughout this section and Section \ref{section6}, 
  the parameters  in $A_g^{\tau,\alpha}$ obey
 \eqref{tau-alpha-3}, i.e.
\begin{equation}
\begin{cases}	  		\tau\leqslant 0  \,& \mbox{ if } \alpha=-1,\\ 		\tau\geqslant 2  \,& \mbox{ if } \alpha=1.  \nonumber 	\end{cases}  \end{equation}
And we denote
  \begin{equation}
	\label{beta-gamma-A2}
	\begin{aligned}
		V[u]=\Delta u g -\varrho\nabla^2 u+\gamma |\nabla u|^2 g +\varrho du\otimes du+A,   \nonumber
	\end{aligned}
\end{equation} 
\begin{equation}
	\label{beta-gamma-A-3}
	\begin{aligned}
		\varrho=\frac{n-2}{\tau-1}, \mbox{ }
		\gamma=\frac{(\tau-2)(n-2)}{2(\tau-1)}, \mbox{ } 
		A=\frac{n-2}{\alpha(\tau-1)} A_{g}^{\tau,\alpha}.  \nonumber 
	\end{aligned}
\end{equation}
  Here one can check $V[u]=\frac{n-2}{\alpha(\tau-1)} A^{\tau,\alpha}_{\tilde{g}}, \mbox{ } \tilde{g}=e^{2u}g$.


 \vspace{1mm}
 Under the assumption  \eqref{tau-alpha-3}, we have the following lemma.
 
 \begin{lemma}
 	Under the assumption  \eqref{tau-alpha-3}, 
 	we have
 	\begin{equation}
 		\label{gammarho0}
 		\gamma=\frac{(n-2)(\tau-2)}{2(\tau-1)}\geqslant0, \quad \gamma+\varrho=\frac{\tau(n-2)}{2(\tau-1)} \geqslant0.
 	\end{equation}
 	
 \end{lemma}

The following lemma asserts that a compact manifold with boundary carries a Morse function without critical points. The construction of such a Morse function is standard
in differential topology.  

  \begin{lemma}
  	\label{lemma-diff-topologuy}
  	Let $(M,g)$ be a compact connected Riemannian manifold of dimension $n\geq 2$ with smooth boundary. Then there is a smooth Morse function $v$ without critical points, that is $d v\neq 0$ in $\bar M$.
  \end{lemma}
\begin{proof} 
	Let $X$ be the double of $M$. Let $w$ be a smooth Morse function on $X$ with the critical set $\{p_i\}_{i=1}^{m+k}$, among which $p_1,\cdots, p_m$ are all the critical points  being in $\bar M$. 
Pick $q_1, \cdots, q_m\in X\setminus \bar M$ but not the critical point of $w$. By homogeneity lemma (see e.g. \cite{Milnor-1997}), one can find a diffeomorphism
$h: X\to X$,	which is smoothly isotopic to the identity, such that 
\begin{itemize}
	\item $h(p_i)=q_i$, $1\leqslant i\leqslant m$.
	\item $h(p_i)=p_i$, $m+1\leqslant i\leqslant m+k$.
\end{itemize}
Then $v=w\circ h^{-1}\big|_{\bar M}$ is the desired Morse function.
	
	\end{proof}

  \begin{proposition}
  	\label{lemma5-main}
  	For  $(\alpha,\tau)$ satisfying \eqref{tau-alpha-3}, 
  	there exists a smooth conformal  admissible metric  
  	on $\bar M$.
  \end{proposition}
  
  \begin{proof}
 According to Lemma \ref{lemma-diff-topologuy}, there exists a smooth   function $v$ with 
  \[ |\nabla v|^2\geqslant a_0  \mbox{ in } \bar M \]
  for some positive constant $a_0$. 
   Without loss of generality,  $v\geqslant0$.
   
Set $ \underline{u}=e^{Nv},$ $\underline{g}=e^{2\underline{u}}g,$
  	then \begin{equation}
  		\begin{aligned}
  			V[\underline{u}]= 
  			N^2e^{Nv}\left((\Delta v g-\varrho\nabla^2v)/N+(1+\gamma e^{Nv})|\nabla v|^2 g+\varrho (e^{Nv}-1)dv\otimes dv\right)+A. \nonumber
  		\end{aligned}
  	\end{equation}	
 
  	The discussion consists of two cases:
  	
  	Case 1: $\varrho> 0$. In this case $$(1+\gamma e^{Nv})|\nabla v|^2 g+\varrho (e^{Nv}-1)dv\otimes dv\geqslant  |\nabla v|^2 g.$$

  	Case 2: $\varrho<0$. In this case $$(1+\gamma e^{Nv})|\nabla v|^2 g+\varrho (e^{Nv}-1)dv\otimes dv\geqslant \left((1-\varrho)+(\gamma+\varrho)e^{Nv}\right)|\nabla v|^2 g\geqslant |\nabla v|^2 g.$$
  	Here we use \eqref{gammarho0}.
  	Therefore,   $\lambda(g^{-1}V[\underline{u}])\in \Gamma$ in $\bar M$ for $N\gg1$. That is, $\underline{g}=e^{2\underline{u}}g$ is an admissible metric.
  	
  \end{proof}

\begin{remark}
	By Corollary \ref{coro3-ingamma}, the construction also works if $(\gamma, \cdots,\gamma,\gamma+\varrho)\in\overline{\Gamma}$. 
\end{remark}

  \medskip
  \section{Asymptotic behavior and uniqueness of solutions}
  \label{section6}
  
  
  In this section we   
 further  assume   $f$ satisfies  
 \eqref{homogeneous-1}, i.e.,
  \begin{equation}
  	\begin{aligned}
  		f(t\lambda)=tf(\lambda),    
  		\mbox{   } \forall \lambda\in\Gamma, \mbox{  } t>0, 
  		\mbox{ with normalization } f(\vec{\bf 1})=1.  \nonumber
  	\end{aligned}
  \end{equation} 
  Then the equation \eqref{main-equ1}  
  reads as follows
  \begin{equation}
  	\label{mainequ-02-2-1}
  	\begin{aligned}
  		f(\lambda(g^{-1}V[u]))	=\frac{(n-2)\psi}{\alpha(\tau-1)} e^{2  u}. 
  	\end{aligned}
  \end{equation}

  For the complete conformal metrics satisfying \eqref{main-equ1}, we prove the following asymptotic behavior and uniqueness of metrics,
  when the prescribed function is a constant when restricted to boundary. Let $\sigma$ denote as in \eqref{distance-function} the distance function.
  \begin{theorem}
  	\label{thm1-unique}
  	Let $(M,g)$ be a compact connected Riemannian manifold of dimension $n\geqslant3$ with smooth boundary.
  	Suppose in addition that \eqref{concave}, \eqref{homogeneous-1-buchong2}, \eqref{homogeneous-1},    \eqref{tau-alpha} and \eqref{tau-alpha-3} hold.
  	Then for any positive smooth function with $\psi|_{\partial M}\equiv1$, 
  	there is a unique smooth complete admissible metric  $\widetilde{g}=e^{2\widetilde{u}_\infty}g$ satisfying \eqref{main-equ1}. Moreover,
  	\begin{equation}
  		\label{asymptotic-rate3}
  		\begin{aligned}
  			\lim_{x\rightarrow\partial M} (\widetilde{u}_\infty(x)+\log \mathrm{\sigma}(x))=\frac{1}{2}\log\frac{\alpha (n\tau+2-2n)}{2(n-2)}. \nonumber
  		\end{aligned}
  	\end{equation} 
  	
  \end{theorem}
  

  We now summarize a theorem due to Aviles-McOwen, extending a famous result of   Loewner-Nirenberg to general Riemannian manifolds.
  \begin{theorem}
  	[\cite{Aviles1988McOwen}]
  	\label{thm-AM} 
  	Let $(M,g)$ be a compact connected Riemannian manifold of dimension $n\geqslant 3$ with smooth boundary. 
  	There is a smooth complete conformal metric with scalar curvature $-1$.
  \end{theorem} 
  The theorem  yields that there is a $\tilde{u}\in C^\infty(M)$ such that
  \begin{equation}
  	\label{scalar-equ1}
  	\begin{aligned}
  		\,& 2(n-1)\Delta \tilde{u}+ (n-1)(n-2)|\nabla \tilde{u}|^2-R_g=e^{2\tilde{u}} \mbox{ in } M, 
  		\,& \lim_{x\rightarrow \partial M} \tilde{u}(x)=+\infty. 
  	\end{aligned}
  \end{equation}
  Geometrically,    $e^{2\tilde{u}}g$  
  is a smooth complete  metric with 
  constant scalar curvature $-1$.


  For conformal scalar curvature equation \eqref{scalar-equ1} 
  on $M=\Omega$ a smooth bounded domain of Euclidean spaces,
  Loewner-Nirenberg \cite{Loewner1974Nirenberg} 
  proved 
  the asymptotic ratio
    \begin{equation}
  	\label{asymptotic-rate1-002}
  	\begin{aligned}
  		\lim_{x\to\partial M}\left(\tilde{u}(x)+\log\mathrm{\sigma}(x) \right)=\frac{1}{2} \log{n(n-1)}.
  	\end{aligned}
  \end{equation}
  This asymptotic property 
can be extended to general Riemannian manifolds, as pointed out by McOwen 
 in \cite{McOwen1993}

  \vspace{1mm}
  Theorem \ref{thm1-unique} follows from the  propositions below.
  
  \begin{proposition}
  	\label{super123}
  	Let $\widetilde{g}_\infty=e^{2\widetilde{u}_\infty}g$
  	be a complete  metric  
  	obeying the equation  \eqref{main-equ1} 
  	with \eqref{concave}, \eqref{homogeneous-1-buchong2}, \eqref{homogeneous-1}, \eqref{tau-alpha} and \eqref{tau-alpha-3}, then
  	\begin{equation} 
  		\label{asymptotic-rate1}
  		\begin{aligned}
  			\lim_{x\rightarrow\partial M} (\widetilde{u}_\infty(x)+\log \mathrm{\sigma}(x))\leqslant \frac{1}{2}\log\frac{\alpha(n\tau+2-2n)}{2(n-2) \inf_{\partial M}\psi}. 
  		\end{aligned}
  	\end{equation}
  \end{proposition}

  \begin{proof}
  	We utilize an approximation in analogy with that used in proof of Theorem \ref{thm2-pde}. 
  	Let $\Omega_\delta$ be as in \eqref{omega-delta}, and let $$\Omega_{\delta,\delta'}=\{x\in M: 0<\delta'<\mathrm{\sigma}(x)<\delta-\delta'\}.$$
  	We choose  $0<\delta'<\frac{\delta}{2}\ll1$ such that $\Omega_\delta$ and $\Omega_{\delta,\delta'}$ are both smooth.  
  	
  	We use Theorem \ref{thm-AM} repeatedly.
  	By Theorem \ref{thm-AM},  $\Omega_{\delta}$ admits a smooth complete conformal metric $g_\infty^{\delta}=e^{2w_\infty^{\delta}}g$ with constant scalar curvature $-1$.
  	Again, 
  	there is a smooth complete metric $g_\infty^{\delta,\delta'}=e^{2u_\infty^{\delta,\delta'}}g$  on $\Omega_{\delta,\delta'}$ with constant scalar curvature $-1$, that is,
  	$$\mathrm{tr} (g^{-1} V[{u_\infty^{\delta,\delta'}}])=\frac{ n\tau+2-2n }{2(n-1)(\tau-1)}e^{2u_\infty^{\delta,\delta'}}
  	\mbox{ in } \Omega_{\delta,\delta'}, 
  	\quad {u_\infty^{\delta,\delta'}}|_{\partial \Omega_{\delta,\delta'}}=+\infty.$$
  	By Lemma \ref{lemma-add2}, $\frac{ n\tau+2-2n }{2(n-1)(\tau-1)}>0$.
  	By maximum principle 
  	\begin{equation} \label{buchong78} \begin{aligned}
  			u_\infty^{\delta,\delta'}\geqslant w_\infty^{\delta}
  			\mbox{ in } \Omega_{\delta,\delta'}.
  	\end{aligned}\end{equation}
  	Furthermore, for any $0<\delta'<\delta_0'<\frac{\delta}{2}$, we have $\Omega_{\delta,\delta_0'}\subset \Omega_{\delta,\delta'}$ and
  	\begin{equation} \label{buchong79} \begin{aligned}
  			u_\infty^{\delta,\delta'}\leqslant u_\infty^{\delta,\delta_0'} \mbox{ in } \Omega_{\delta,\delta_0'}.\end{aligned}\end{equation}
    	According to  \eqref{key1-main}, the solution  $\widetilde{u}_\infty$ 
    	satisfies $$\mathrm{tr} (g^{-1} V[\widetilde{u}_\infty]) \geqslant \frac{n(n-2)\psi}{\alpha(\tau-1)} 
  	e^{2\widetilde{u}_\infty}.$$
  	The maximum principle implies 
  	\begin{equation}  \begin{aligned}
  			u_\infty^{\delta,\delta'} \geqslant \widetilde{u}_\infty+\frac{1}{2}\log \left(\frac{2n(n-1)(n-2)}{\alpha(n\tau+2-2n)}\inf_{\Omega_{\delta,\delta'}} \psi \right) 
  			\mbox{ in } \Omega_{\delta,\delta'}. \nonumber
  	\end{aligned}\end{equation}
  	
  Take $u_\infty^{\delta}(x)= \underset{\delta'\rightarrow0}{\lim}\, u_\infty^{\delta,\delta'}(x)$ (such a limit exists by \eqref{buchong78}-\eqref{buchong79}), 
  	then 	\begin{equation} 
  		\label{f-equ1}
  		\begin{aligned}
  			\,& u_\infty^{\delta} \geqslant w_\infty^{\delta}, \,&  u_\infty^{\delta} \geqslant  \widetilde{u}_\infty+  
  			\frac{1}{2}\log \left(\frac{2n(n-1)(n-2)}{\alpha(n\tau+2-2n)}\inf_{\Omega_{\delta}} {\psi}\right)  \mbox{ in } \Omega_{\delta} 
  		\end{aligned}
  	\end{equation}
  	and $e^{2u_\infty^\delta}g$ is a smooth complete conformal metric on $\Omega_{\delta}$ of scalar curvature $-1$. 
  	(We have local estimates since
  	it is of  uniform ellipticity).
  	Furthermore, \eqref{asymptotic-rate1-002} gives
  	$$\lim_{x\rightarrow \partial M} (u_\infty^{\delta}(x)+\log\mathrm{\sigma}(x))= \frac{1}{2}\log [n(n-1)].$$
  	Putting them together,
  	we get \eqref{asymptotic-rate1}. 
  \end{proof}
  
  

  \begin{proposition}
  	\label{proposition-asymptotic}
 Assume \eqref{concave}, \eqref{homogeneous-1-buchong2},
\eqref{homogeneous-1}, \eqref{tau-alpha} and \eqref{tau-alpha-3} hold.
  	Then any admissible complete metric $\widetilde{g}_\infty=e^{2\widetilde{u}_\infty}g$  satisfying \eqref{main-equ1}  
  	obeys
  	\begin{equation}
  		\label{asymptotic-rate2}
  		\begin{aligned}
  			\lim_{x\rightarrow\partial M} (\widetilde{u}_\infty(x)+\log \mathrm{\sigma}(x))\geqslant \frac{1}{2}\log\frac{\alpha(n\tau+2-2n)}{2 (n-2)\sup_{\partial M}\psi}. 
  		\end{aligned}
  	\end{equation}
  	
  \end{proposition}
  
  \begin{proof}
  	The proof is different from that presented in \cite{Loewner1974Nirenberg,Gursky-Streets-Warren2011}.
  	The key ingredients in the proof are \eqref{key1-main} and 
  	$|\nabla \mathrm{\sigma}|=1$ on $\partial M$. 
  	Fix a constant $0<\epsilon<1$. 
  	We consider 
  	\begin{equation}
  		\label{001}
  		\left\{
  		\begin{aligned}
  			\,& f(\lambda(g^{-1} V[u_{k,\epsilon}]))=\frac{(n-2)\psi}{\alpha(\tau-1)} e^{2u_{k,\epsilon}} \mbox{ in } M, \\
  			\,& u_{k,\epsilon}=\log {k} +\frac{1}{2}\log\frac{(1-\epsilon)^2 \alpha(n\tau+2-2n)}{2(n-2)(\sup_{\partial M} {\psi}+\epsilon)} \mbox{ on } \partial M.
  		\end{aligned}
  		\right.
  	\end{equation}
  	The solvability is obtained by Theorem 
  	\ref{thm1-finiteBVC}.
  	The maximum principle implies
  	\begin{equation}
  		\label{b4b}
  		\begin{aligned}
  			\widetilde{u}_\infty\geqslant u_{k,\epsilon} \mbox{ in } M,
  		\end{aligned}
  	\end{equation}
  	\begin{equation}
  		\label{1114-3}
  		\begin{aligned}
  		\inf_M u_{k,\epsilon} \geqslant \frac{1}{2}\min\left\{ \inf_M \log\frac{f(\lambda(g^{-1}A_g^{\tau,\alpha}))}{{\psi}},  \log\frac{k^2(1-\epsilon)^2 \alpha(n\tau+2-2n)}{2(n-2)(\sup_{\partial M} {\psi}+\epsilon)} \right\}. 
  		\end{aligned}
  	\end{equation}

  	The local subsolution is given by
  	\begin{equation}
  		\label{1115-0}
  		\begin{aligned}
  			h_{k,\epsilon,\delta}(\mathrm{\sigma})=\log\frac{k}{k\mathrm{\sigma}+1}
  			+\frac{1}{2}\log\frac{(1-\epsilon)^2 \alpha(n\tau+2-2n)}{2(n-2)(\sup_{\partial M}  {\psi}+\epsilon)}+
  			\frac{1}{\mathrm{\sigma}+\delta}- \frac{1}{\delta}. \nonumber
  		\end{aligned}
  	\end{equation} 
  	The part $``\frac{1}{\mathrm{\sigma}+\delta}-\frac{1}{\delta}"$  is inspired by \cite{Gursky-Streets-Warren2011}.
  	Straightforward computation tells
  	\begin{equation}
  		\label{1114-1}
  		\begin{aligned}
  			\,& h_{k,\epsilon,\delta}'=-\frac{k}{k\mathrm{\sigma}+1}-\frac{1}{(\mathrm{\sigma}+\delta)^2},
  			\,&h_{k,\epsilon,\delta}''=\frac{k^2}{(k\mathrm{\sigma}+1)^2}+\frac{2}{(\mathrm{\sigma}+\delta)^3}, \nonumber
  		\end{aligned}
  	\end{equation}
  	\begin{equation}
  		\begin{aligned}
  			h_{k,\epsilon,\delta}'^2
  			=\frac{k^2}{(k\mathrm{\sigma}+1)^2}+\frac{1}{(\mathrm{\sigma}+\delta)^4}+\frac{2k}{(k\mathrm{\sigma}+1)(\mathrm{\sigma}+\delta)^2}, \nonumber
  		\end{aligned}
  	\end{equation}
  	$$ h_{k,\epsilon,\delta}''+\gamma h_{k,\epsilon,\delta}'^2
  	=\frac{(1+\gamma)   k^2}{(k\mathrm{\sigma}+1)^2}+\frac{2}{(\mathrm{\sigma}+\delta)^3}+\frac{\gamma}{(\mathrm{\sigma}+\delta)^4}+\frac{2k\gamma}{(k\mathrm{\sigma}+1)(\mathrm{\sigma}+\delta)^2}.$$
  	\begin{equation}
  		\label{yuan-39}
  		\begin{aligned}
  			V[h_{k,\epsilon,\delta}] =\,& (h_{k,\epsilon,\delta}''+\gamma h_{k,\epsilon,\delta}'^2)|\nabla \mathrm{\sigma}|^2 g 
  			+ \varrho(h_{k,\epsilon,\delta}'^2-h_{k,\epsilon,\delta}'') d\mathrm{\sigma}\otimes d\mathrm{\sigma}
  			\\\,& +h_{k,\epsilon,\delta}' (\Delta \mathrm{\sigma}\cdot g-\varrho\nabla^2 \mathrm{\sigma})+A.  \nonumber
  		\end{aligned}
  	\end{equation}
  	By
  	\eqref{gammarho0} we have $\varrho+\gamma\geqslant 0$, then 
  	\begin{equation}
  		\label{yuan-310}
  		\begin{aligned}
  			h_{k,\epsilon,\delta}''+\gamma h_{k,\epsilon,\delta}'^2+ \varrho(h_{k,\epsilon,\delta}'^2-h_{k,\epsilon,\delta}'')\geqslant \frac{(1+\gamma)   k^2}{(k\mathrm{\sigma}+1)^2}+\frac{2(1-\varrho)}{(\mathrm{\sigma}+\delta)^3}.
  		\end{aligned}
  	\end{equation}
  	
  	Next we are going to verify that 
  	\begin{claim}
  		If $ \delta\leqslant\frac{1}{4}$  or $ k\geqslant\frac{1}{\delta} $
  		then
  		\begin{equation}
  			\label{yuan-key312}
  			\begin{aligned}
  				V[h_{k,\epsilon,\delta}]
  				\geqslant 
  				\left( \frac{(1+\gamma)k^2}{(k\mathrm{\sigma}+1)^2}+\frac{2}{(\mathrm{\sigma}+\delta)^3} \right)|\nabla \mathrm{\sigma}|^2 g
  				-\left(\frac{k}{k\mathrm{\sigma}+1}+\frac{1}{(\mathrm{\sigma}+\delta)^2}\right)(\Delta \mathrm{\sigma}  g -\varrho\nabla^2\mathrm{\sigma})+A.  \nonumber 
  			\end{aligned}
  		\end{equation} 
  	\end{claim}
  	First, it is easy to see $ h_{k,\epsilon,\delta}'^2- h_{k,\epsilon,\delta}''\geqslant0$ provided $ \delta\leqslant\frac{1}{4}$  or $ k\geqslant\frac{1}{\delta}$.
  So the case  $\varrho>0$ is trivial. When $\varrho<0$, together with \eqref{yuan-310}, we get
  	\begin{equation}
  		\begin{aligned}
  			(h_{k,\epsilon,\delta}''+\gamma h_{k,\epsilon,\delta}'^2)|\nabla \mathrm{\sigma}|^2 g 
  			+ \varrho(h_{k,\epsilon,\delta}'^2-h_{k,\epsilon,\delta}'') d\mathrm{\sigma}\otimes d\mathrm{\sigma}  
  			\geqslant 
  			\left( \frac{(1+\gamma)k^2}{(k\mathrm{\sigma}+1)^2}+\frac{2}{(\mathrm{\sigma}+\delta)^3} \right)|\nabla \mathrm{\sigma}|^2 g.  \nonumber
  		\end{aligned}
  	\end{equation}
  	This completes the proof of the claim. 
  	
  	Note that $|\nabla \mathrm{\sigma}|^2=1$ on ${\partial M}$, 
  	$\Delta\mathrm{\sigma} \cdot g-\varrho\nabla^2\mathrm{\sigma}$ is bounded in $\bar M$,
  	$\frac{k}{k\mathrm{\sigma}+1} \geqslant \frac{1}{\mathrm{\sigma}+\delta} \mbox{ for } k\geqslant\frac{1}{\delta},$
  	and $\frac{1}{\mathrm{\sigma}+\delta}\gg1$ in $\Omega_\delta$ 
  	if $0<\delta\ll1$.
  	For $\epsilon$ fixed, there is a $\delta_\epsilon$ depending on $\epsilon$ and other known data  
  	such that, in $\Omega_\delta$ for 
  	$0<\delta<\delta_{\epsilon}$,
  	\begin{enumerate}
  		\item $\mathrm{\sigma}(x)$ is smooth, and $|\nabla \mathrm{\sigma}|^2\geqslant 1-\epsilon$.
  		\item $\frac{2}{(\mathrm{\sigma}+\delta)^3} |\nabla \mathrm{\sigma}|^2 g
  		-  \frac{1}{(\mathrm{\sigma}+\delta)^2} (\Delta \mathrm{\sigma} \cdot g -\varrho\nabla^2\mathrm{\sigma})+A\geqslant0$.
  		\item $\frac{k^2\epsilon(1+\gamma)}{(k\mathrm{\sigma}+1)^2} |\nabla \mathrm{\sigma}|^2 g-\frac{k}{k\mathrm{\sigma}+1} (\Delta \mathrm{\sigma} \cdot g -\varrho\nabla^2\mathrm{\sigma})\geqslant0$.
  		\item $\sup_{\partial M} {\psi}+\epsilon\geqslant \sup_{\Omega_{\delta}} {\psi}$.
  	\end{enumerate}
  	From now on we fix $0<\delta<\delta_\epsilon$ (for example $\delta=\frac{\delta_\epsilon}{2}$) and $k\geqslant\frac{1}{\delta}$.
  	Putting the discussion above together, in $\Omega_\delta$,  we have
  	\begin{equation}
  		\label{buchong1}
  		\begin{aligned}
  			f(\lambda(g^{-1}V[h_{k,\epsilon,\delta}]))
  			\geqslant 
  			f\left(\frac{k^2{(1+\gamma)(1-\epsilon)} }{(k\mathrm{\sigma}+1)^2}\cdot|\nabla \mathrm{\sigma}|^2 I_n\right) 
  			\geqslant  \frac{(n-2)\psi}{\alpha(\tau-1)}   e^{2h_{k,\epsilon,\delta}}.   
  		\end{aligned}
  	\end{equation}
  	Note that   $u_{k,\epsilon}=h_{k,\epsilon,\delta}$ on $\partial M$, and
  	$u_{k,\epsilon}|_{\mathrm{\sigma}=\delta} \geqslant h_{k,\epsilon,\delta}(\delta)$ (according to \eqref{1114-3}).
  	Here we also use $\frac{\log\frac{k}{k\delta+1}}{\frac{1}{2\delta}}\leqslant\frac{\log\frac{1}{\delta}}{\frac{1}{2\delta}}\rightarrow 0$ as $\delta\rightarrow0^+$.
  	Then the maximum principle yields 
  	\begin{equation}
  		\begin{aligned}
  			u_{k,\epsilon}\geqslant h_{k,\epsilon,\delta}=\log\frac{k}{k\mathrm{\sigma}+\delta}+\frac{1}{2}\log\frac{(1-\epsilon)^2\alpha(n\tau+2-2n)}{2(n-2)(\sup_{\partial M} {\psi}+\epsilon)}+\frac{1}{\mathrm{\sigma}+\delta}-\frac{1}{\delta}  \nonumber
  		\end{aligned}
  	\end{equation}
  	on $\Omega_{\delta}$.
  	From \eqref{b4b}, let $k\rightarrow +\infty$,  
  	we know that on $\Omega_\delta$,
  	\begin{equation}
  		\begin{aligned}
  			\widetilde{u}_\infty(x)+\log \mathrm{\sigma}(x)\geqslant \frac{1}{2}\log\frac{(1-\epsilon)^2
  				\alpha(n\tau+2-2n)}{2(n-2)(\sup_{\partial M} {\psi}+\epsilon)}
  			+\frac{1}{\mathrm{\sigma}+\delta}-\frac{1}{\delta}.  \nonumber
  		\end{aligned}
  	\end{equation}
  	Thus \eqref{asymptotic-rate2} follows.

  \end{proof}

  \begin{proposition}
  	\label{unique-prop}
  Suppose	\eqref{concave}, \eqref{homogeneous-1-buchong2}, \eqref{tau-alpha}, \eqref{homogeneous-1} and \eqref{tau-alpha-3} hold.
  	Let $u_\infty(x)=\underset{k\to+\infty}{\lim}\, u^{(k)}(x),$
  	where 
  	\begin{equation}	\begin{aligned}		\,& f(\lambda(g^{-1}V[u^{(k)}]))=\frac{(n-2)\psi}{\alpha(\tau-1)} e^{2u^{(k)}} 
  			\mbox{ in } M,
  			\,& u^{(k)}=\log k \mbox{ on } \partial M. \nonumber
  	\end{aligned}	\end{equation}
 For  any admissible solution $\widetilde{u}_\infty$  to 
   	the Dirichlet problem for equation \eqref{mainequ-02-2-1} with infinite boundary value condition, then 
  	$$u_\infty\leqslant \widetilde{u}_\infty \leqslant u_\infty +\frac{1}{2}(\sup_{\partial M}\log\psi- \inf_{\partial M}\log\psi) \mbox{ in } M.$$
  	
  \end{proposition}

  \begin{proof}
  	The comparison principle yields $\widetilde{u}_\infty\geqslant u^{(k)}$ $(\forall k\geqslant1)$. Then $\widetilde{u}_\infty\geqslant u_\infty$.
  	
  	Next, we prove the second inequality. 
  	Note that
  	\eqref{asymptotic-rate1} and  \eqref{asymptotic-rate2} imply $$\lim_{x\rightarrow\partial M}(\widetilde{u}_\infty(x)-u_\infty(x))\leqslant\frac{1}{2}(\sup_{\partial M}\log\psi-\inf_{\partial M}\log\psi).$$
  	Assume 
  	there is an  interior point $x_0\in M$ such that
  	$$(\widetilde{u}_\infty-u_\infty)(x_0)=\sup_M(\widetilde{u}_\infty-u_\infty)>\frac{1}{2}(\sup_{\partial M}\log\psi-\inf_{\partial M}\log\psi).$$ 
  	This is a contradiction since the maximum principle yields $\widetilde{u}_\infty(x_0)\leqslant u_\infty(x_0)$.  
  \end{proof}

    \medskip
  
  \section{Local and boundary estimates for   nonlinear uniformly elliptic equations}
  \label{section4}
  


  This section is devote to deriving the local and boundary estimates for a more general equation
  \begin{equation}
  	\label{hessianequ2-riemann}
  	\begin{aligned}
  		F(\nabla^2 u+A(x,u,\nabla u)):=f\left(\lambda(g^{-1}(\nabla^2 u+A(x,u,\nabla u)))\right) =\psi(x,u,\nabla u)
  	\end{aligned}
  \end{equation} 
  which is of fully uniform ellipticity,  
   i.e., 
 \begin{equation} 	\begin{aligned}	f_{i}(\lambda)\geqslant \theta\sum_{j=1}^n f_j(\lambda)>0 	\mbox{ in } \Gamma,		\quad \forall 1\leqslant i\leqslant n,  \nonumber  	\end{aligned}  \end{equation}
  where 
  $\psi(x,z,p)$ and $A(x,z,p)$ 
  denote respectively smooth function and  symmetric $(0,2)$-type tensor 
  of variables $(x,z,p)$,  
  $(x,p)\in T M$,
  $z\in \mathbb{R}$. 
  
  Again, we say $u $ is an \textit{admissible} function for \eqref{hessianequ2-riemann} if $$\lambda(g^{-1}(\nabla^2 u+A(x,u,\nabla u)))\in \Gamma.$$
  Throughout this section, for $u$ we denote
  \[\mathfrak{g}[u]=\nabla^2 u+A(x,u,\nabla u), 
  \quad \psi[u]=\psi(x,u,\nabla u).\]
  Moreover,  for any admissible solution $u$ to \eqref{hessianequ2-riemann},
   $\mathfrak{g}=\mathfrak{g}[u]$, $\lambda=\lambda(g^{-1}\mathfrak{g})$, $F^{ij}=\frac{\partial F}{\partial a_{ij}}(\mathfrak{g})$.
  Then the matrix $\{F^{ij}\}$ (with respect to $\{g^{ij}\}$) has eigenvalues $f_1,\cdots, f_n$. 
   Moreover,   $\sum_{i=1}^n f_i=F^{ij}g_{ij},   \sum_{i=1}^n f_i\lambda_i =F^{ij}\mathfrak{g}_{ij}.$
  

   \vspace{1mm}
  The local estimates for second derivatives can be stated as follows.
  \begin{theorem}
  	\label{interior-2nd-2} 
  	Let $B_r\subset M$ be a geodesic ball of radius $r$.
  	Let $u\in C^4(B_{r})$ be an admissible solution to equation \eqref{hessianequ2-riemann} in $B_r$.
  	Assume 
  	\eqref{concave}, \eqref{addistruc} 
  	and \eqref{fully-uniform2} hold.
  	Then  
  	\begin{equation}
  		\begin{aligned}
  			\sup_{B_{{r}/{2}}}|\nabla^2 u| \leqslant C, \nonumber
  		\end{aligned}
  	\end{equation}
  	where $C$ depends on  $|u|_{C^1(B_r)}$, $r^{-1}$, 
  	$\theta^{-1}$ and other known data.
  \end{theorem}
  
  We will prove local gradient estimate under asymptotic assumptions:
  \begin{equation}
  	\label{key-assimption1}
  	\begin{aligned}
  		\,& |D_pA|\leqslant \gamma(x,z) |p|, 
  		\,&
  		D_zA+\frac{1}{|p|^2}g(\nabla' A,p)\leqslant \beta(x,z,|p|) |p|^2g, \\
  		\,& |D_p \psi|\leqslant \gamma(x,z) |p|, 
  		\,&
  		-D_z\psi-\frac{1}{|p|^2}g(\nabla' \psi,p)\leqslant \beta(x,z,|p|) |p|^2,
  	\end{aligned}
  \end{equation}
  where
  $\gamma=\gamma(x,z)$ and  $\beta=\beta(x,z,r)$  are positive continuous functions with
  \begin{equation}\begin{aligned}
  		\,& \lim_{r\rightarrow+\infty} \beta(x,z,r)=0, \,& (x,z,r)\in \bar M\times \mathbb{R}\times [0,+\infty).\nonumber
  	\end{aligned}
  \end{equation} 
  Here $\nabla' \psi$, $\nabla' A$ and $D_pA$ are as in Introduction, 
  and  the meanings of $D_zA$, $D_z\psi$  and $D_p\psi$ are similar.
  
  \begin{theorem}
  	\label{thm-gradient2}
  	Let $B_r\subset M$ as above and 
  	$u\in C^3( B_{r})$ be an admissible solution to  \eqref{hessianequ2-riemann} in $B_r$.
  	In addition to 
  	\eqref{concave}, 
  	we assume  \eqref{fully-uniform2} and \eqref{key-assimption1} hold.
  	Suppose 
  	there is a positive constant $\kappa_0$ depending not on $\nabla u$ such that
  	\begin{equation}
  		\label{sumfi-assumption1}
  		\begin{aligned}
  			F^{ij}g_{ij}\geqslant \kappa_0.
  		\end{aligned}
  	\end{equation}
  	Then there is a positive constant $C$ depending on $|u|_{C^0(B_r)}$,   $\theta^{-1}$ and other known data such that
  	$$\sup_{B_{{r}/{2}}}|\nabla u|\leqslant  {C}/{r}.$$
  	
  \end{theorem}
  
  \begin{remark}
  	By Lemma \ref{lemma3.4} and \eqref{concavity1}, if 
  	\eqref{concave} and \eqref{addistruc} hold then
  	\begin{equation}
  		\label{sumfi}
  		\begin{aligned}
  			\sum_{i=1}^n f_i(\lambda)> \kappa:={(f(R_0\vec{\bf 1})-\psi[u])}/{R_0}>0 \mbox{ for some } R_0>0.
  		\end{aligned}
  	\end{equation}
  	Moreover, if $f$ is homogeneous of degree one, then   
  	$\sum_{i=1}^n f_i(\lambda)\geqslant f(\vec{\bf 1})>0.$ 
  	
  	As a result,
  	the assumption \eqref{sumfi-assumption1}  holds  when $f$ is a homogeneous function  of degree one, or  
  	$f$ satisfies \eqref{addistruc} and $\psi[u]=\psi(x,u)$. 
  \end{remark}

  The boundary estimates for second derivatives are also obtained. 
  \begin{theorem}
  	\label{thm2-bdy}
  	Let $(M,g)$ be a compact Riemannian manifold with smooth boundary.
  	Suppose  
  	\eqref{concave}, \eqref{addistruc} 
  	and \eqref{fully-uniform2} hold.
  	Let $u\in C^3(M)\cap C^2(\bar M)$ be an admissible solution to  
  	\eqref{hessianequ2-riemann} with 
  	$u=\varphi$ on $\partial M$ for $\varphi\in C^3(\bar M)$, then $$\sup_{\partial M}|\nabla^2 u|\leqslant C$$ holds for 
  	a uniform  constant depending on $\theta^{-1}$, $|u|_{C^1(\bar M)}$, $|\varphi|_{C^3(\bar M)}$ and other known data. 
  	
  \end{theorem}

  \subsection{Useful computation}
  

  Let $e_1,...,e_n$ be the local frame 
  as above. We denote 
  $A_{ij}(x,u,\nabla u)=A(x,u,\nabla u)(e_i,e_j)$,
  and $\mathfrak{g}_{ij}=\nabla_{ij}u+A_{ij}(x,u,\nabla u)$.
  The linearized operator of \eqref{hessianequ2-riemann} is given by
  \begin{equation}
  	\label{linearized-operator}
  	\begin{aligned}
  		\mathcal{L}v=F^{ij}\nabla_{ij}v+(F^{ij}A_{ij,p_l}-\psi_{p_l})\nabla_l v \mbox{ for } v\in C^2(\bar M). 
  	\end{aligned}
  \end{equation}

  Similar computation can be found in \cite{Guan2015Jiao}.
  We have
  \begin{equation}
  	\begin{aligned}
  		\nabla_k\mathfrak{g}_{ij}=
  		\nabla_{kij}u+\nabla'_kA_{ij} +A_{ij,z} \nabla_k u+A_{ij,p_l}\nabla_{kl}u. \nonumber
  	\end{aligned}
  \end{equation}
  By differentiating equation \eqref{hessianequ2-riemann},
  \begin{equation}
  	\label{differ-equa1}
  	\begin{aligned}
  		F^{ij}\nabla_k\mathfrak{g}_{ij}=\nabla'_k\psi+\psi_z\nabla_k u+\psi_{p_l}\nabla_{kl}u,
  	\end{aligned}
  \end{equation}
  \begin{equation}
  	\label{differ-buchong3}
  	\begin{aligned}
  		F^{ii}\nabla_{11}\mathfrak{g}_{ii}=\,&
  		\nabla'_{11}\psi
  		-F^{ij,kl}\nabla_1\mathfrak{g}_{ij} \nabla_1\mathfrak{g}_{kl}
  		+ 2\nabla'_1\psi_z \nabla_1 u
  		+\psi_z\nabla_{11}u
  		+2\nabla'_1\psi_{p_l}\nabla_{1l}u 
  		\\\,&
  		+2\psi_{zp_l}\nabla_{1l}u\nabla_1u
  		+\psi_{zz}|\nabla_1u|^2
  		+\psi_{p_lp_m}\nabla_{1l}u\nabla_{1m}u
  		+\psi_{p_l}\nabla_{11l}u.  
  	\end{aligned}
  \end{equation}
  Using  \eqref{differ-equa1}, $\nabla_{ij}u=\nabla_{ji}u$ and $\nabla_{ijk}u=\nabla_{kij}u-R_{ljik}\nabla_l u$,
  \begin{equation}
  	\label{ineq1-c1}
  	\begin{aligned}
  	\,& F^{ij}\nabla_{ijk}u+(F^{ij}A_{ij,p_l}-\psi_{p_l})\nabla_{lk}u  
 		\\=\,&
  	\nabla'_k \psi+\psi_z \nabla_k u 	-F^{ij}(R_{ljik}\nabla_l u+\nabla'_kA_{ij}+A_{ij,z}\nabla_k u).
  	\end{aligned}
  \end{equation}

  Let $w=|\nabla u|^2$.
  The straightforward computation yields
  \begin{equation}
  	\label{wi1}
  	\begin{aligned}
  		\,& \nabla_i w=2\nabla_k u \nabla_{ik}u,
  		\,& \nabla_{ij}w=2\nabla_{ik}u\nabla_{jk}u+2\nabla_k u\nabla_{ijk}u,
  	\end{aligned}
  \end{equation}
  \begin{equation}
  	\label{Lw1}
  	\begin{aligned}
  		\mathcal{L}w=\,&
  		2F^{ij}\nabla_{ik}u \nabla_{jk}u
  		-2F^{ij}R_{ljik}\nabla_k u \nabla_l u
  		+2\nabla'_k\psi \nabla_k u
  		\\\,&
  		+2\psi_z|\nabla u|^2
  		-2F^{ij}(\nabla'_k A_{ij}\nabla_k u+A_{ij,z}|\nabla u|^2).
  	\end{aligned}
  \end{equation}

  
  \subsection{Local gradient estimate}
  Note that the equation \eqref{hessianequ2-riemann} is of fully uniform ellipticity, the proof of local estimates is standard.
  For convenience, we present the details below. 
  As above we denote $w=|\nabla u|^2$. 
  We consider the quantity $$Q:=\eta w e^\phi$$ where $\phi$ is determined later. 
  Following \cite{Guan2003Wang} 
  let $\eta$ be a smooth function 
  with compact support in $B_r \subset M$ and 
  \begin{equation}
  	\label{eta1}
  	\begin{aligned}
  		0\leqslant \eta\leqslant 1, \mbox{  } \eta|_{B_{\frac{r}{2}}}\equiv1, 
  		\mbox{  } |\nabla\eta|\leqslant   {C\sqrt{\eta}}/{r},
  		\mbox{  } |\nabla^2\eta|\leqslant {C}/{r^2}.
  	\end{aligned}
  \end{equation}
  The quantity $Q$ attains its maximum at an interior point $x_0\in M$. 
  We may assume $|\nabla u|(x_0)\geqslant 1$.
  By maximum principle at $x_0$
  \begin{equation}
  	\label{mp1}
  	\begin{aligned}
  		\,& \frac{\nabla_i\eta}{\eta}+\frac{\nabla_i w}{w}+\nabla_i\phi=0, 
  		\,& \mathcal{L}(\log\eta+\log w+\phi)\leqslant 0.
  	\end{aligned}
  \end{equation}
  Around $x_0$ we choose a local orthonormal frame $e_1,\cdots, e_n$;
  for simplicity,  we further assume $e_1,\cdots,e_n$ have been chosen such that at $x_0$,  
  $\mathfrak{g}_{ij}$ is diagonal (so is $F^{ij}$).
  
  \vspace{1mm}
 \noindent{\em Step 1. Computation and estimation for $\mathcal{L}(\log w)$.}
  By \eqref{wi1},
  \begin{equation}
  	\label{wiwi1}
  	\begin{aligned}
  		F^{ij}\nabla_iw \nabla_j w \leqslant 4|\nabla u|^2 F^{ij}\nabla_{ik}u\nabla_{jk}u. 
  	\end{aligned}
  \end{equation}
  Using Cauchy-Schwarz inequality, one derives by \eqref{mp1}
  \begin{equation}
  	\label{wiwi2}
  	\begin{aligned}
  		F^{ij} \frac{\nabla_i w\nabla_j w}{w^2}\leqslant (1+\epsilon)
  		\left(\frac{1}{\epsilon}F^{ij}\frac{\nabla_i \eta\nabla_j \eta}{ \eta^2}
  		+F^{ij}\nabla_i \phi \nabla_j \phi \right), \nonumber
  	\end{aligned}
  \end{equation}
  which, together with \eqref{wiwi1}, yields 
  \begin{equation}
  	\label{wiwi3}
  	\begin{aligned}
  		F^{ij} \frac{\nabla_i w\nabla_j w}{w^2}
  		\leqslant (1-\epsilon^2)
  		\left(\frac{1}{\epsilon}F^{ij}\frac{\nabla_i \eta\nabla_j \eta}{ \eta^2}
  		+F^{ij}\nabla_i \phi \nabla_j \phi \right)+\frac{4\epsilon}{w}F^{ij}\nabla_{ik}u \nabla_{jk}u. \nonumber
  	\end{aligned}
  \end{equation}
  Set $0<\epsilon\leqslant \frac{1}{4}$. We now obtain 
  \begin{equation}
  	\label{Llogw}
  	\begin{aligned}
  		\mathcal{L}(\log w) \geqslant \,&
  		\frac{1}{w}F^{ij}\nabla_{ik}u \nabla_{jk}u
  		+2(\psi_z+\frac{1}{w}\nabla'_k \psi \nabla_k u)
  		-2F^{ij}(A_{ij,z}+\frac{1}{w}\nabla'_k A_{ij}\nabla_k u)  
  		\\
  		\,&
  		-\frac{1-\epsilon^2}{\epsilon}F^{ij}\frac{\nabla_i \eta\nabla_j \eta}{ \eta^2}
  		-(1-\epsilon^2)F^{ij}\nabla_i\phi \nabla_j\phi -C_0\sum F^{ii}.
  	\end{aligned}
  \end{equation}
  
\noindent{\em Step 2. Construction and computation of $\phi$.}
  As in \cite{Guan2008IMRN},  
  let $\phi=v^{-N}$ where $v=u-\inf_{B_r} u+2$ ($v\geqslant 2$ in $B_r$)  and $N\geqslant 1$ to be chosen later. By direct computation,
  \begin{equation}
  	\label{Lphi}
  	\begin{aligned}
  		\mathcal{L}\phi
  		= \,& N(N+1)v^{-N-2}F^{ij}\nabla_i u \nabla_j u-Nv^{-N-1}F^{ij}\nabla_{ij}u 
  		\\\,&
  		-Nv^{-N-1}(F^{ij}A_{ij,p_l}-\psi_{p_l})\nabla_l u. 
  	\end{aligned}
  \end{equation} 
  
  \subsubsection*{Step 3. Completion of the proof.}
  By \eqref{fully-uniform2},
  \begin{equation}
  	\label{fully-application1}
  	\begin{aligned}
  		\,& F^{ij}\nabla_{ik}u \nabla_{jk}u \geqslant \theta |\nabla_{ik} u|^2 \sum F^{ii}, 
  		\,& F^{ij}\nabla_i u \nabla_j u \geqslant \theta w\sum F^{ii}.
  	\end{aligned}
  \end{equation}
  By Cauchy-Schwarz inequality again, one derives
  \begin{equation}
  	\label{cauchy-1}
  	\begin{aligned}
  		\frac{1}{8}N^2 v^{-N-2}\theta w+\frac{\theta}{2w}|\nabla_{ik}u|^2 \geqslant \frac{1}{2}N\theta |\nabla_{ik}u| v^{-\frac{N}{2}-1}. \nonumber
  	\end{aligned}
  \end{equation}
  
  We choose $N\gg1$ so that 
  \begin{equation}
  	\begin{aligned}
  		\,& \frac{1}{4}N\theta v^{-\frac{N}{2}-1}\geqslant Nv^{-N-1}, 
  		\,& N(N+1)v^{-N-2}-N^2v^{-2N-2}\geqslant \frac{N^2}{2}v^{-N-2}.  \nonumber
  	\end{aligned}
  \end{equation}
  Suppose furthermore that $\frac{1}{8}Nv^{-N-2}\theta w\geqslant C_0$ where $C_0$ is as in \eqref{Llogw} (otherwise we are done). 
  Together with \eqref{mp1}, \eqref{Llogw}, \eqref{Lphi} and \eqref{fully-application1},
  we derive
  \begin{equation}
  	\label{key1}
  	\begin{aligned}
  		0 \geqslant \,&
  		\frac{1}{4}\theta N^2 w v^{-N-2}\sum F^{ii}
  		+\frac{1}{4}N\theta |\nabla_{ik} u|v^{-\frac{N}{2}-1} \sum F^{ii}
  		+2(\psi_z+\frac{1}{w}\nabla'_k \psi \nabla_k u)
  		\\\,&
  		-2F^{ij}(A_{ij,z}+\frac{1}{w}\nabla'_k A_{ij}\nabla_k u)  
  		-Nv^{-N-1}(F^{ij}A_{ij,p_l}-\psi_{p_l})\nabla_l u
  		\\ \,&
  		+(F^{ij}A_{ij,p_l}-\psi_{p_l})\frac{\nabla_l \eta}{\eta} + \frac{F^{ij}\nabla_{ij}\eta}{\eta}
  		-\frac{1+\epsilon-\epsilon^2}{\epsilon}F^{ij}\frac{\nabla_i \eta\nabla_j \eta}{ \eta^2}. \nonumber
  	\end{aligned}
  \end{equation}
  Using \eqref{eta1} and the asymptotic assumption \eqref{key-assimption1}, we obtain
  \begin{equation}
  	\begin{aligned}
  		0\geqslant  \,&
  		\frac{1}{4}\theta N^2 w v^{-N-2}\sum F^{ii}
  		+\frac{1}{4}N\theta |\nabla_{ik} u|v^{-\frac{N}{2}-1} \sum F^{ii}  -\frac{C}{r^2\eta}\sum F^{ii}
  		\\\,&
  		-\left(\beta(x,u,|\nabla u|) w+CNv^{-N-1}w+\frac{C\sqrt{w}}{r\sqrt{\eta}}\right) \left(1+\sum F^{ii}\right), \nonumber
  	\end{aligned}
  \end{equation}
  which gives $\eta w\leqslant \frac{C}{r^2}$, as required.
  

  \subsection{Local estimate for second derivatives}
  
  Let's consider 
  $$P(x)=\max_{\xi\in T_x\bar M, |\xi|=1} \eta\mathfrak{g}(\xi,\xi)e^{\varphi},
  $$
  where $\eta$ is the cutoff function as given by \eqref{eta1} and $\varphi$ is a function to be chosen later.
  One knows $P$ achieves maximum at an interior point $x_0\in B_r$ and for $\xi\in T_{x_0}\bar M$. 
  Around $x_0$ we choose a smooth orthonormal local frame $e_1,\cdots, e_n$ such that $e_1(x_0)=\xi$, $\Gamma_{ij}^k(x_0)=0$ and
  $\{\mathfrak{g}_{ij}(x_0)\}$ is diagonal (so is $\{F^{ij}(x_0)\}$). We may assume $\mathfrak{g}_{11}(x_0)\geqslant 1$.
  At $x_0$ one has
  \begin{equation}
  	\label{mp2nd-1}
  	\begin{aligned}
  		\,& \nabla_i \varphi+\frac{\nabla_i \mathfrak{g}_{11}}{\mathfrak{g}_{11}}+\frac{\nabla_i \eta}{\eta}, 
  		\,&
  		\mathcal{L}(\varphi+\log\eta+\log\mathfrak{g}_{11}) \leqslant 0.
  	\end{aligned}
  \end{equation}
  
\noindent{\em Step 1. Estimation for $\mathcal{L}(\mathfrak{g}_{11})$.}
  From \eqref{differ-buchong3}, we get
  \begin{equation}
  	\label{key2nd-1}
  	\begin{aligned}
  		F^{ii}\nabla_{11}\mathfrak{g}_{ii}\geqslant \psi_{p_l}\nabla_l \mathfrak{g}_{11}-C_0\mathfrak{g}_{11}^2.
  	\end{aligned}
  \end{equation}
  The straightforward computation shows
  \begin{equation}
  	\label{key2nd-2}
  	\begin{aligned}
  		F^{ii}(\nabla_{ii}\mathfrak{g}_{11}-\nabla_{11}\mathfrak{g}_{ii}) \geqslant 
  		-F^{ii}A_{ii,p_l}\nabla_l\mathfrak{g}_{11}
  		-C_2\mathfrak{g}_{11}^2 \sum_{i=1}^n F^{ii}-C_2\mathfrak{g}_{11}.
  	\end{aligned}
  \end{equation}
  Here we use \eqref{differ-equa1}. 
  Combing \eqref{key2nd-1} and \eqref{key2nd-2}, 
  \begin{equation}
  	\label{key2nd-3}
  	\begin{aligned}
  		\mathcal{L}(\mathfrak{g}_{11})=
  		F^{ii}\nabla_{ii}\mathfrak{g}_{11}+(F^{ii}A_{ii,p_l}-\psi_{p_l})\nabla_l\mathfrak{g}_{11}  \nonumber
  		\geqslant
  		-C_3\mathfrak{g}_{11}^2(1+\sum_{i=1}^n F^{ii}).
  	\end{aligned}
  \end{equation}
  Using \eqref{mp2nd-1} and Cauchy-Schwarz inequality, 
  \begin{equation}
  	\begin{aligned}
  		F^{ii}\frac{|\nabla_i\mathfrak{g}_{11}|^2}{\mathfrak{g}_{11}^2}
  		\leqslant \frac{3}{2}F^{ii}|\nabla_i\varphi|^2+3F^{ii}\frac{|\nabla_i \eta|^2}{\eta^2}.  \nonumber
  	\end{aligned}
  \end{equation}
  Hence
  \begin{equation}
  	\label{key2nd-4}
  	\begin{aligned}
  		\mathcal{L}(\log\mathfrak{g}_{11}) \geqslant -C_3\mathfrak{g}_{11}(1+\sum_{i=1}^n F^{ii}) -\frac{3}{2}F^{ii}|\nabla_i\varphi|^2
  		-3F^{ii}\frac{|\nabla_i \eta|^2}{\eta^2}.  \nonumber
  	\end{aligned}
  \end{equation}
  
\noindent{\em Step 2. Construction and estimate for $\mathcal{L}\varphi$.}
  Following \cite{Guan2008IMRN}  we set $$\varphi=\varphi(w)=(1-\frac{w}{2N})^{-\frac{1}{2}} \mbox{ where $w=|\nabla u|^2$ and $N=\sup_{\{\eta>0\}}|\nabla u|^2$}. $$
  One can check $\varphi'=\frac{\varphi^3}{4N},$ $\varphi''=\frac{3\varphi^5}{16N^2}$ and $1\leqslant \varphi\leqslant \sqrt{2}$.
  By \eqref{Lw1}, \eqref{fully-uniform2} and Cauchy-Schwarz inequality, we obtain
  \begin{equation}
  	\begin{aligned}
  		\mathcal{L}w \geqslant 
  		\frac{1}{2}F^{ii}\mathfrak{g}_{ii}^2-C_4 \sum_{i=1}^n F^{ii}
  		\geqslant (\frac{\theta }{2}\mathfrak{g}_{11}^2 -C_4)\sum_{i=1}^n F^{ii}.  \nonumber
  	\end{aligned}
  \end{equation}
  Consequently, 
  \begin{equation}
  	\label{key2nd-5}
  	\begin{aligned}
  		\mathcal{L}\varphi
  		\geqslant 
  		\frac{3\varphi^5}{16N^2}F^{ii}|\nabla_i w|^2+
  		\frac{\varphi^3}{4N}  (\frac{\theta }{2}\mathfrak{g}_{11}^2 -C_4)\sum_{i=1}^n F^{ii}.  \nonumber
  	\end{aligned}
  \end{equation}
  
  \noindent{\em Step 3. Completion of the proof.}
  From \eqref{eta1},
  \begin{equation}
  	\begin{aligned}
  		\,&
  		F^{ii}\frac{|\nabla_i\eta|^2}{\eta^2} \leqslant \frac{C}{r^2 \eta}\sum_{i=1}^n F^{ii},
  		\,&
  		\mathcal{L}(\log\eta)\geqslant -\frac{C}{r^2 \eta}\sum_{i=1}^n F^{ii} -\frac{C}{r\sqrt{\eta}}.
  		 \nonumber
  	\end{aligned}
  \end{equation}
  Next,
  \begin{equation}
  	\begin{aligned}
  		\frac{3}{2}F^{ii}|\nabla_i\varphi|^2=\frac{3\varphi^6}{32N^2}F^{ii}|\nabla_i w|^2
  		\leqslant \frac{3\varphi^5}{16N^2}F^{ii}|\nabla_i w|^2.  \nonumber
  	\end{aligned}
  \end{equation}
  Finally, at $x_0$  
  \begin{equation}
  	\begin{aligned}
  		0\geqslant
  		  \frac{3}{4N} \left(\frac{\theta}{2} \mathfrak{g}_{11}^2-\frac{4NC_3}{3}\mathfrak{g}_{11}-C_4\right)\sum_{i=1}^n F^{ii} 
  		-\frac{C}{r^2 \eta}\sum_{i=1}^n F^{ii} 
  		-\frac{C}{r\sqrt{\eta}} 
  		-C_3\mathfrak{g}_{11}. \nonumber
  	\end{aligned}
  \end{equation}
  Combining with \eqref{sumfi} we have $\eta \mathfrak{g}_{11}\leqslant \frac{C}{r^2}$ at $x_0$. 
  Therefore the proof is complete.

  \subsection{Boundary estimate}
  
  For a point $x_0\in\partial M$ we choose a smooth orthonormal local frame $e_1, \cdots, e_n$ around $x_0$ such that
  $e_n$ is unit outer normal vector field when restricted to $\partial M$.  
We denote $\rho(x):=\mathrm{dist}_g(x,x_0)$  the distance function from $x$ to $x_0$
  with respect to $g$,
 and $M_\delta:=\{x\in M: \rho(x)<\delta\}.$
  As in \eqref{distance-function}, $\mathrm{\sigma}(x)$ is  the distance from $x$ to $\partial M$.
  We know that
  $\mathrm{\sigma}(x)$ is smooth and $|\nabla \mathrm{\sigma}|\geqslant \frac{1}{2}$ in $M_\delta$ for small $\delta$.

  \subsubsection*{Case 1. Pure tangential derivatives.}   
  From the boundary value condition, 
  \begin{equation}
  	\label{morry-1}
  	\begin{aligned}
  		\,& \nabla_\alpha u=\nabla_\alpha \varphi, 
  		\,& \nabla_{\alpha\beta}u=\nabla_{\alpha\beta}\varphi+\nabla_{n}(u-\varphi)\mathrm{II}(e_\alpha,e_\beta)
  		\mbox{ on } \partial M  \nonumber
  	\end{aligned}
  \end{equation}
  for $1\leqslant \alpha, \beta<n$, where $\mathrm{II}$ is the fundamental form of $\partial M$. 
  This gives the bound of
  second estimates for pure tangential derivatives
  \begin{equation}
  	\label{ineq2-bdy}
  	\begin{aligned}
  		|\mathfrak{g}_{\alpha\beta}(x_0)|\leqslant C.
  	\end{aligned}
  \end{equation}
  
  \subsubsection*{Case 2. Mixed derivatives.}
  For $1\leqslant\alpha<n$, the local barrier function 
 is given by
  \begin{equation}
  	\begin{aligned}
  		\Psi=\pm \nabla_\alpha (u-\varphi) 
  		+A_1 \left(\frac{N \mathrm{\sigma}^2}{2} -t \mathrm{\sigma}\right)
  		-A_2\rho^2+A_3 \sum_{k=1}^{n-1} |\nabla_k (u-\varphi)|^2 
  		\mbox{ in } M_\delta,  \nonumber
  	\end{aligned}
  \end{equation}
  where $A_1$, $A_2$, $A_3$, $N$, $t$ are all positive constants to be determined, and $N\delta-2t\leqslant 0.$
  The construction of  local barrier is standard.
  
  
  To deal with local barrier functions, we need the following formula 
  \begin{equation}
  	\label{ineq1-bdy}
  	\begin{aligned}
  		\nabla_{ij}(\nabla_k u)=\nabla_{ijk}u+\Gamma_{ik}^l\nabla_{jl}u+\Gamma_{jk}^l\nabla_{il}u+\nabla_{\nabla_{ij}e_k}u,\nonumber
  	\end{aligned}
  \end{equation}
 see e.g. \cite{Guan12a}.
  Combining with \eqref{ineq1-c1} one derives
  \begin{equation}
  	\label{bdy-1}
  	\begin{aligned}
  		|\mathcal{L}(\nabla_k (u-\varphi))|\leqslant C\left(1+\sum_{i=1}^n f_i+\sum_{i=1}^n f_i |\lambda_i|\right).  \nonumber
  	\end{aligned}
  \end{equation}
  As in \eqref{Lw1}, we can prove
  \begin{equation}
  	\label{bdy-2}
  	\begin{aligned}
  		\mathcal{L}(|\nabla_k (u-\varphi)|^2) \geqslant F^{ij}\mathfrak{g}_{ik}\mathfrak{g}_{jk} 
  		-C\left(1+\sum_{i=1}^n f_i+\sum_{i=1}^n f_i |\lambda_i|\right).  \nonumber
  	\end{aligned}
  \end{equation}

  By \cite[Proposition 2.19]{Guan12a}, there exists an index $1\leqslant r\leqslant n$ with
  \begin{equation}
  	\begin{aligned}
  		\sum_{k=1}^{n-1} F^{ij}\mathfrak{g}_{ik}\mathfrak{g}_{jk} \geqslant \frac{1}{2} \sum_{i\neq r} f_i\lambda_i^2.  \nonumber
  	\end{aligned}
  \end{equation}

  It follows from    \eqref{concavity1} and Lemma \ref{lemma3.4} that
  $$0<\sum_{i=1}^n f_i \lambda_i \leqslant \psi(x,u,\nabla u)-f(\vec{\bf 1})+\sum_{i=1}^nf_i.$$
  Combining with Cauchy-Schwarz inequality and
  $$\sum_{i=1}^n f_i |\lambda_i| =2\sum_{\lambda_i\geqslant 0} f_i\lambda_i -\sum_{i=1}^n f_{i} \lambda_i = \sum_{i=1}^n f_i\lambda_i -2\sum_{\lambda_i<0} f_{i} \lambda_i,$$ 
   we have
  \begin{equation}
  	\label{flambda}
  	\begin{aligned}
  		\sum_{i=1}^n f_i |\lambda_i|
  		\leqslant   \epsilon \sum_{i\neq r} f_i\lambda_i^2  +\frac{C}{\epsilon}\sum_{i=1}^n f_i+C. \nonumber
  	\end{aligned}
  \end{equation}

  Let $\kappa$ be as in \eqref{sumfi}.
  If $\delta $ and $t$ are chosen small enough such that
  \begin{equation}
  	\begin{aligned}
  		\mathcal{L}\left(\frac{N \mathrm{\sigma}^2}{2}-t \mathrm{\sigma}\right)
  		\geqslant 
  		\frac{N \kappa \theta}{8(1+\kappa)}
  		\left(1+\sum_{i=1}^n f_i\right).  \nonumber
  	\end{aligned}
  \end{equation}

  Putting those inequalities together, if $A_1\gg A_2$, $A_1\gg A_3>1$ then 
  \begin{equation}
  	\begin{aligned}
  		\mathcal{L}(\Psi)
  		>0 \mbox{ in } M_\delta.  \nonumber
  	\end{aligned}
  \end{equation}
  On the other hand $\Psi|_{\partial M}=-A_2\rho^2\leqslant 0$; while on $\partial M_\delta\setminus \partial  M$, 
   $\Psi\leqslant 0$ if  $N\delta-2t\leqslant 0$, $A_2\gg A_1$.
  Note that $\Psi(x_0)=0$. 
  Thus
  \begin{equation}
  	\label{mix-1}
  	\begin{aligned}
  		|\mathfrak{g}_{\alpha n}(x_0)|\leqslant C \mbox{ for } 1\leqslant \alpha\leqslant n-1.
  	\end{aligned}
  \end{equation}
  
  \subsubsection*{Case 3. Double normal derivatives.}
  Fix $x_0\in\partial M$. 
  Since $\mathrm{tr}(g^{-1}\mathfrak{g})>0$,  \eqref{ineq2-bdy} and \eqref{mix-1}, we have $\mathfrak{g}_{nn}\geqslant -C$. 
  We assume 
  $\mathfrak{g}_{nn}(x_0)\geqslant 1$ (otherwise we are done).
  According to \eqref{fully-uniform2}, 
  $F^{nn}\geqslant \theta \sum_{i=1}^n f_i. $
  By the concavity of equation,   \eqref{ineq2-bdy} and \eqref{mix-1}, 
  \begin{equation}
  	\begin{aligned}
  		\psi[u](x_0)-F(g) 
  		\geqslant 
  		F^{ij} (\mathfrak{g}_{ij}-\delta_{ij}) 
  		\geqslant 
  		-C'\sum_{i=1}^n f_i (\lambda)+  (\theta\mathfrak{g}_{nn}-1)  \sum_{i=1}^n f_i(\lambda).  \nonumber
  	\end{aligned}
  \end{equation}
  This gives $\mathfrak{g}_{nn}(x_0) \leqslant C.$ Here we also use \eqref{sumfi}.

  \subsubsection*{Acknowledgements}  The author wishes to express his gratitude to Professor Robert McOwen for sending his beautiful paper \cite{McOwen1993}. The author would  like to thank Professor Yi Liu for 
  answering questions 
  related to
 the proof of Lemma \ref{lemma-diff-topologuy}.
 The author also wishes to thank Dr. Ze Zhou for useful discussions on the homogeneity lemma.
  The author is supported by the National Natural Science  Foundation of China through grant 11801587.


\bigskip



\begin{thebibliography}{99}
	\medskip

\bibitem{Aubin1976}  T. Aubin,   
\textit{\'Equations diff\'erentielles non lin\'eaires et probl\`eme de Yamabe concernant la courbure scalaire}, {J. Math. Pures  Appl.} {\bf 55}  (1976),  269--296.

\bibitem{Aviles1985McOwen}
P. Aviles,  and R. McOwen,
{\em Conformal deformations of complete manifolds with negative curvature},
{J. Differential Geom.} {\bf 21} (1985), 269--281. 


\bibitem{Aviles1988McOwen2}
P. Aviles,  and R. McOwen,
{\em Conformal deformation to constant negative scalar curvature on noncompact Riemannian manifolds},
{J. Differential Geom.} {\bf 27} (1988), 225--239.

\bibitem{Aviles1988McOwen}
P. Aviles,  and R. McOwen,
\textit{Complete conformal metrics with negative scalar curvature in compact Riemannian manifolds},
{Duke Math. J.} {\bf 56} (1988),  395--398. 

\bibitem{Besse1987}
A. Besse,
{Einstein manifolds}. Classics in Mathematics. Springer-Verlag, Berlin, 2008. 

\bibitem{Branson2006Gover} 
T. Branson, and A. Gover, \textit{Variational status of a class of fully nonlinear curvature prescription problems}, Calc. Var. 
PDE. {\bf32} (2008),   253--262.

\bibitem{Brendle2014Chen}
S. Brendle,  and S.-Y. S. Chen, 
{\em An existence theorem for the Yamabe problem on manifolds with boundary},
J. Eur. Math. Soc. (JEMS) {\bf16}  (2014), 991--1016.

\bibitem{CNS3}
L. Caffarelli, L. Nirenberg, and J. Spruck, \textit{The Dirichlet problem for nonlinear second-order elliptic equations III: Functions of eigenvalues of the Hessians},
{Acta Math.} {\bf 155} (1985),  261--301.




\bibitem{ChangGurskyYang2002} 
S.-Y.  A.  Chang, M. Gursky, and P. Yang,
\textit{An equation of Monge-Amp\`ere type in conformal geometry, and four-manifolds of positive Ricci curvature}, Ann. Math.  {\bf 155} (2002), 709--787.

 

\bibitem{SChen2007} S.-Y. S. Chen, \textit{Boundary value problems for some fully nonlinear elliptic equations}, Calc. Var. PDE.  {\bf30} (2007), 1--15.


\bibitem{SChen2009}	S.-Y. S. Chen,	\textit{Conformal deformation on manifolds with boundary}, Geom. Funct. Anal. {\bf19} (2009), 1029--1064.


 

\bibitem{Cherrier1984}
P. Cherrier, 
\textit{Probl\`emes de Neumann non lin\'eaires sur les vari\'et\'es riemanniennes}, J. Funct. Anal. {\bf57} (1984), 154--206.

\bibitem{Escobar1992JDG}
J. F. Escobar, 
{\em The Yamabe problem on manifolds with boundary}, J. Differential Geom. {\bf35}  (1992), 21--84.

\bibitem{Escobar1992Ann}
J. F. Escobar, 
{\em Conformal deformation of a Riemannian metric to a scalar flat metric with constant mean curvature on the boundary},
Ann. Math. {\bf136}   (1992), 1--50.
 

\bibitem{Evans82} L. C. Evans, \emph{Classical solutions of fully nonlinear convex, second order elliptic equations}, {Comm. Pure Appl. Math.} {\bf 35} (1982),  333--363.



\bibitem{FuShengYuan}
J.-X. Fu, W.-M. Sheng, and L.-X. Yuan,
{\em Prescribed $k$-curvature problems on complete noncompact Riemannian manifolds}, 
Int. Math. Res. Not.   
{2020}, no. 23, 
9559--9592.

 \bibitem{FuWangWuFormtype2010} J.-X. Fu,  Z.-Z. Wang, and D.-M. Wu, \textit{Form-type Calabi-Yau equations}, {Math. Res. Lett.} {\bf 17} (2010), 887--903.








\bibitem{Ge2006Wang} Y.-X. Ge, and G.-F. Wang, \textit{On a fully nonlinear Yamabe problem},  Ann. Sci. \'{E}ole Norm. Sup. {\bf 39} (2006), 569--598.

\bibitem{GT1983} D. Gilbarg, and N. Trudinger,  {Elliptic Partial Differential Equations of Second Order}. Classics in Mathematics, Springer-Verlag, Berlin, reprint  of the 1998 Edition, 2001.

\bibitem{GLN-2018}
M. Gonz\'alez,   Y.-Y. Li, and L. Nguyen,  
{\em Existence and uniqueness to a fully nonlinear version of the Loewner-Nirenberg problem}, Comm. Math. Stat. {\bf 6} (2018), 269--288.


 

\bibitem{Guan2007AJM} B. Guan, \textit{Conformal metrics with prescribed curvature functions on manifolds with boundary}, Amer. J. Math. {\bf129} (2007),  915--942.


\bibitem{Guan2008IMRN}
B. Guan,  \textit{Complete conformal metrics of negative Ricci curvature on compact manifolds with boundary}, Int. Math. Res. Not. IMRN 2008, Art. ID rnn 105, 25 pp.

\bibitem{Guan12a} B. Guan, \textit{Second order estimates and regularity for fully nonlinear elliptic equations on Riemannian manifolds}, {Duke Math. J.} {\bf 163} (2014),  1491--1524.





\bibitem{Guan2015Jiao} B. Guan, and H.-M. Jiao, \textit{Second order estimates for Hessian type fully nonlinear elliptic equations on Riemannian manifolds}, Calc. Var. PDE. {\bf 54} (2015), 2693--2712.

 
 

\bibitem{Guan1991Spruck}	B. Guan, and J. Spruck,  {\em Interior gradient estimates for solutions of prescribed curvature equations of parabolic type},	Indiana Univ. Math. J. {\bf40} (1991),  1471--1481.
 

\bibitem{Guan2003Wang} P.-F. Guan, and G.-F. Wang, \textit{Local estimates for a class of fully nonlinear equations arising from conformal geometry}, Int. Math. Res. Not. {2003},   1413--1432.

\bibitem{Guan2003Wang-CrelleJ} P.-F. Guan, and G.-F. Wang, 
\textit{A fully nonlinear conformal flow on locally conformally flat manifolds},
J. Reine Angew. Math. {\bf557} (2003), 219--238.

 
\bibitem{Gursky-Streets-Warren2010}
M. Gursky, J. Streets,  and M. Warren,
\textit{Conformally bending three-manifolds with boundary}, Ann. Inst. Fourier (Grenoble) {\bf60} (2010), 
2421--2447. 

\bibitem{Gursky-Streets-Warren2011}
M. Gursky, J. Streets,  and M. Warren,
\textit{Existence of complete conformal metrics of negative Ricci curvature on manifolds with boundary},
Calc. Var.  
PDE. {\bf41} (2011),  21--43. 

 



\bibitem{Gursky2003Viaclovsky} 	M. Gursky, and J.  Viaclovsky, \textit{Fully nonlinear equations on Riemannian manifolds with negative curvature}, Indiana Univ. Math. J. {\bf52} (2003),   399--419. 

\bibitem{Gursky2007Viaclovsky} M. Gursky, and J.  Viaclovsky, \textit{Prescribing symmetric functions of the eigenvalues of the Ricci tensor}, Ann. Math. {\bf 166} (2007), 475--531.


 
\bibitem{Han1999Li}
Z.-C. Han, and  Y.-Y. Li, 
{\em The Yamabe problem on manifolds with boundary: existence and compactness results}, Duke Math. J. 
{\bf99}  (1999),   489--542.
 

 \bibitem{Harvey2011Lawson} F. Harvey, and H. Lawson, 
\textit{Dirichlet duality and the nonlinear Dirichlet problem on Riemannian manifolds}, J. Differential Geom. {\bf88} (2011), 395--482. 

 

\bibitem{Jin1988}
Z.-R. Jin, 
{A counterexample to the Yamabe problem for complete noncompact manifolds}.  
 Lecture Notes in Math. vol. 1306, pp. 93--101, Springer-Verlag, Berlin and New York, 1988.
 
\bibitem{Jin1993}
Z.-R. Jin, 
{\em Prescribing scalar curvatures on the conformal classes of complete metrics with negative curvature}, 
{Trans. Amer. Math. Soc.} {\bf 340} (1993), 785--810.
 
 
 
\bibitem{Krylov83} N. V. Krylov,
\emph{Boundedly nonhomogeneous elliptic and parabolic equations in a domain}, {Izvestia Math. Ser.} {\bf 47} (1983), 75--108.



\bibitem{ABLi2003YYLi} A.-B. Li, and Y.-Y. Li, \textit{On some conformally invariant fully nonlinear equations}, Comm. Pure Appl. Math. {\bf56} (2003), 1416--1464.



\bibitem{Li2011Sheng}
Q.-R. Li, and W.-M. Sheng,
\textit{Some Dirichlet problems arising from conformal geometry}, Pacific J. Math. 
{\bf 251} (2011), 337--359.
 

\bibitem{LiYY1991} Y.-Y. Li, \textit{Interior gradient estimates for solutions of certain fully nonlinear elliptic equations}, J. Differential Equations {\bf90} (1991),   172--185.

\bibitem{Lin1994Trudinger}
M. Lin, and N. Trudinger,
\textit{On some inequalities for elementary symmetric functions}, Bull. Austral. Math. Soc. {\bf 50} (1994), 317--326.







 

\bibitem{Loewner1974Nirenberg}
C. Loewner, and L. Nirenberg, {Partial differential equations invariant under conformal or projective transformations}. Contributions to analysis, pp.
 245--272. Academic Press, New York, 1974. 
 
 \bibitem{Lohkamp-1} J. Lohkamp,  {\em Metrics of negative Ricci curvature}, Ann. Math.   {\bf 140} (1994), 
 655--683.
 
  \bibitem{Lohkamp-2}
 J. Lohkamp, 
 {\em Negative bending of open manifolds},
 J. Differential Geom. {\bf40}  (1994), 461--474. 

\bibitem{Marques2005}
F. C. Marques, 
{\em Existence results for the Yamabe problem on manifolds with boundary},
Indiana Univ. Math. J. {\bf54}  (2005),  1599--1620.

\bibitem{Marques2007}
F. C. Marques,
{\em Conformal deformations to scalar-flat metrics with constant mean curvature on the boundary},
Comm. Anal. Geom. {\bf15} (2007),  381--405.

\bibitem{McOwen1993}
R. McOwen,
{Singularities and the conformal scalar curvature equation}. Geometric analysis and nonlinear partial differential equations (Denton, TX, 1990), pp. 221--233, Lecture Notes in Pure and Appl. Math. 144, Dekker, New York, 1993.

\bibitem{Milnor-1997}
J. W. Milnor,  Topology from the differentiable viewpoint. Princeton University Press, Princeton, NJ, 1997.

\bibitem{Ni-82Invent}
W.-M. Ni,  {\em On the elliptic equation $\Delta u+K(x)e^{2u}=0$   and conformal metrics with prescribed Gaussian curvatures}, Invent. Math. {\bf66} (1982), 
343--352. 

\bibitem{Ni-82Indiana}
W.-M. Ni, 
{\em On the elliptic equation $\Delta u+K(x)u^{(n+2)/(n-2)}=0$, its generalizations, and applications in geometry}, 
Indiana Univ. Math. J. {\bf31} (1982), 493--529. 
 

\bibitem{Schoen1984} R. Schoen,  {\em Conformal deformation of a Riemannian metric to constant scalar curvature}, {J. Differential Geom.} {\bf20}  (1984), 479--495.

\bibitem{ShaJP1986} J.-P. Sha,  \textit{$p$-convex Riemannian manifolds}, Invent. Math. {\bf83} (1986),  437--447.

\bibitem{ShengTrudingerWang2007} W.-M. Sheng, N. Trudinger, and X.-J. Wang, \textit{The Yamabe problem for higher order curvatures}, J. Differential Geom. {\bf 77} (2007), 515--553.

\bibitem{ShengUrbasWang-Duke} 	W.-M. Sheng, J. Urbas, and X.-J. Wang,  	{\em Interior curvature bounds for a class of curvature equations}, {Duke Math. J.} {\bf 123} (2004),  235--264.

\bibitem{Sheng2006Zhang} W.-M. Sheng,  and Y. Zhang, 	
{\em A class of fully nonlinear equations arising from conformal geometry}, Math. Z. {\bf255} (2007), 17--34.



\bibitem{Sui2017JGA} 
Z.-N. Sui, 
{\em Complete conformal metrics of negative Ricci curvature on Euclidean spaces},
{J. Geom. Anal.} {\bf27} (2017), 893--907. 


 \bibitem{GTW15} 	G. Sz\'{e}kelyhidi, V.  Tosatti,  and {B. Weinkove}, \textit{Gauduchon metrics with prescribed volume form},  Acta Math. {\bf 219} (2017),  181--211.



\bibitem{Trudinger1968} N. Trudinger, 
{\em Remarks concerning the conformal deformation of Riemannian structures on compact manifolds},
Ann. Scuola Norm. Sup. Pisa Cl. Sci. (3) {\bf22}  (1968), 265--274.

\bibitem{Trudinger90} N.  Trudinger, \emph{The Dirichlet problem for the prescribed curvature equations,} Arch. Rational Mech. Anal. {\bf111} (1990),  153--179.

 

\bibitem{Urbas2002}	{J. Urbas},	 
{Hessian equations on compact Riemannian manifolds}. Nonlinear problems in mathematical physics and related topics, II, pp. 367–377, Int. Math. Ser. (N. Y.), 2, Kluwer/Plenum, New York, 2002.

\bibitem{Viaclovsky2000}
J. Viaclovsky,
\textit{Conformal geometry, contact geometry, and the calculus of variations}, Duke Math. J. {\bf101} (2000),   283--316. 

\bibitem{WuHH-1987}
H.-H. Wu,  \textit{Manifolds of partially positive curvature}, Indiana Univ. Math. J. {\bf36} (1987),  525--548.


 



\bibitem{yuan-regular-DP} R.-R. Yuan, \textit{On the regularity of Dirichlet problem for fully non-linear elliptic equations on Hermitian manifolds}, Preprint. arXiv:2203.04898v1.
 
 
  
  \end{thebibliography}
\end{document}